\documentclass[12pt]{amsart}

\usepackage[hmargin=0.8in,height=8.6in]{geometry}
\usepackage{amssymb,amsthm, times}
\usepackage{delarray,verbatim}
\usepackage{ifpdf}
\ifpdf
\usepackage[pdftex]{graphicx}
 \else
\usepackage[dvips]{graphicx}
 \fi

\usepackage{bm}

\usepackage{ifpdf}
\usepackage{color}
\definecolor{webgreen}{rgb}{0,.5,0}
\definecolor{webbrown}{rgb}{.8,0,0}
\definecolor{emphcolor}{rgb}{0.95,0.95,0.95}
\usepackage{comment} 

\usepackage{hyperref}
\hypersetup{%
          colorlinks=true,
          linkcolor=webbrown,
          filecolor=webbrown,
          citecolor=webgreen,
          breaklinks=true}
\ifpdf \hypersetup{pdftex,
            pdfstartview=FitH, 
            bookmarksopen=true,
            bookmarksnumbered=true
} \else \hypersetup{dvips} \fi

\newcommand {\ud}{{\rm d}}

\numberwithin{equation}{section}

\newtheorem{theorem}{Theorem}[section]
\newtheorem{proposition}{Proposition}[section]
\newtheorem{corollary}{Corollary}[section]
\newtheorem{remark}{Remark}[section]
\newtheorem{lemma}{Lemma}[section]

\newtheorem{assump}{Assumption}[section]

\numberwithin{remark}{section} \numberwithin{proposition}{section}
\numberwithin{corollary}{section}
\newcommand {\R}{\mathbb{R}}

\newcommand {\p}{\mathbb{P}}

\newcommand {\E}{\mathbb{E}}

\newcommand{\diff}{{\rm d}}

\newcommand{\lev}{L\'{e}vy }

\newcommand{\e}{\mathbb{E}}

\begin{document}
	\title[Hybrid cont. and per. barrier strat. in the dual model:
	optimality and fluctuation identities]{Hybrid continuous and periodic barrier strategies in the dual model:
		optimality and fluctuation identities}
	
	\thanks{This version: \today.  J. L. P\'erez  is  supported  by  CONACYT,  project  no.\ 241195.
K. Yamazaki is in part supported by MEXT KAKENHI grant no.\  26800092 and 17K05377.}
\author[J. L. P\'erez]{Jos\'e-Luis P\'erez$^*$}
\thanks{$*$\, Department of Probability and Statistics, Centro de Investigaci\'on en Matem\'aticas A.C. Calle Jalisco s/n. C.P. 36240, Guanajuato, Mexico. Email: jluis.garmendia@cimat.mx.  }
\author[K. Yamazaki]{Kazutoshi Yamazaki$^\dag$}
\thanks{$\dag$\, Department of Mathematics,
Faculty of Engineering Science, Kansai University, 3-3-35 Yamate-cho, Suita-shi, Osaka 564-8680, Japan. Email: kyamazak@kansai-u.ac.jp.   }
\date{}
	\maketitle
	
	
		\begin{abstract}  
Avanzi et al.\ \cite{ATW2} recently studied an optimal dividend problem where dividends are paid
both periodically and continuously with different  transaction costs.  
In the Brownian model with Poissonian periodic dividend payment opportunities, they showed that the optimal strategy is either of the pure-continuous, pure-periodic, or hybrid-barrier type.  In this paper, we generalize the results of their previous study to the dual (spectrally positive L\'evy) model.  The optimal strategy is again of the hybrid-barrier type and can be concisely expressed using the scale function.  These results are confirmed through a sequence of numerical experiments. \\
\noindent \small{\noindent  AMS 2010 Subject Classifications: 60G51, 93E20, 91B30 
}\\
	\end{abstract}
	
	\section{Introduction}

In de Finetti's optimal dividend problems, the aim is to maximize the expected net present value (NPV) of the dividends accumulated until ruin. 
In other words, the objective is to find the optimal dividend strategy, namely the one that always strikes a
 balance between maximizing dividends and minimizing the risk of ruin. 
 Because the surplus process is typically assumed to be a time-homogeneous Markov process,  the optimal strategy can reasonably be assumed
 to be of \emph{barrier type}: dividends are paid such that the surplus does not exceed a certain level.  Given this conjecture, one can follow the classical 
 ``\emph{guess and verify}" procedure: the candidate barrier (or free boundary) is chosen by the continuous/smooth fit principle, and  the optimality is shown via verification arguments.  While the optimality of a barrier strategy can fail in some cases (see \cite{Loeffen2008}), all existing explicit solutions are,
 to the best of our knowledge, barrier strategies  or variations thereof.

In this paper, we consider the version of the optimal dividend problem considered in \cite{ATW2}, where dividends are paid both periodically and continuously.  The periodic dividend decision times arrive at the jump times of an independent Poisson process, as described in \cite{ATW} and \cite{YP2016}, and, in addition, one has the option to pay dividends at any time but at a different (proportional) transaction cost.  Avanzi et al.\ \cite{ATW2} studied the case wherein the surplus is driven by a (drifted) Brownian motion.  In this paper, we show that analogous results can be attained when the surplus process is generalized to be a spectrally positive \lev process.

The objective is to show the optimality of the \emph{hybrid (continuous and periodic) barrier strategy}.
Namely, for some periodic barrier $a$ and continuous barrier $b$ such that $0 \leq a \leq b$, 
the surplus process is pushed down to $a$ whenever it is above $a$ at the periodic dividend decision times, while it is also pushed down continuously so that it does not go above $b$ uniformly in time.  In particular, this strategy reduces to the \emph{pure continuous barrier strategy} when $a = b$, and to the \emph{pure periodic barrier strategy}  when $b = \infty$.  For a more detailed discussion of this strategy, see \cite{ATW2}. 



To solve this problem, we take the following steps:
\begin{enumerate}
\item First, the expected NPVs of the (both periodically and continuously collected)  dividends under the hybrid barrier strategy for any choice of $a$ and $b$ are computed.   The controlled surplus process under the hybrid barrier strategy is the dual of the \emph{Parisian reflected process} studied in Avram et al.\ \cite{APY} with additional classical reflection caused by the  continuous barrier strategy.  We shall extend these results to obtain semi-analytical expressions for the expected NPV of the dividends via the scale function.

\item Second, the values for the periodic and continuous barriers, defined as $a^*$ and $b^*$, respectively, are selected.  As this combines the classical singular control case \cite{BKY} and the periodic case \cite{YP2016}, the value function is clearly expected to satisfy analogous smoothness conditions at these two barriers.  

More precisely, at the level $b^*$ where the process is \emph{reflected in the classical sense} \cite{BKY}, the value function is continuously differentiable (resp.\ twice continuously differentiable) when the driving \lev process is of bounded (resp.\ unbounded) variation.  In contrast, at the level $a^*$ where the process is \emph{reflected in the Parisian sense}, as in \cite{YP2016}, it is twice continuously differentiable (resp.\ thrice continuously differentiable) if it is of bounded (resp.\ unbounded) variation. See \cite{Yamazaki_smooth} and \cite{Yamazaki_contraction} for similar conditions in the optimal stopping case.

These conditions at $b^*$ and $a^*$ provide us with two equations, defined as $\mathbf{C}_b$ and $\mathbf{C}_a$, respectively.  Under the assumption that the driving process drifts to infinity, we shall show that
either there exists a pair $0 < a^* < b^*$ or $0 = a^* < b^*$ corresponding to the case of a hybrid barrier strategy, or that we should choose a pure continuous strategy (i.e.,  $a^* = b^*$) or a pure periodic barrier strategy (i.e., $b^* = \infty$).  When $0 < a^* < b^*$, $\mathbf{C}_b$ and $\mathbf{C}_a$ are simultaneously satisfied, whereas if  $0 = a^* < b^*$, $\mathbf{C}_b$ and a weaker version of $\mathbf{C}_a$ (with the equality replaced by an inequality) are satisfied.  The cases $a^* = b^*$ and $b^* = \infty$ are treated separately.


\item Finally, optimality is verified.
Again, because the problem is the combination of the classical singular and periodic cases, the candidate value function is expected to  simultaneously satisfy the sufficient conditions given in \cite{BKY} and \cite{YP2016}.  We show that these conditions are indeed sufficient, and that the candidate value function, with the barriers selected in the previous steps, satisfies these conditions.

\end{enumerate}




To confirm the analytical results obtained, we also conduct numerical experiments for the case driven by the phase-type \lev process as described in \cite{Asmussen_2004}. In this case, the scale function can be represented as  a linear combination of (complex) exponentials (see \cite{Egami_Yamazaki_2010_2}), and hence the optimal barriers $a^*$ and $b^*$ as well as the value function can be computed instantaneously.  We illustrate this computation process, and confirm both optimality and the convergence to the cases wherein pure continuous/periodic barrier strategies are optimal.



To the best of our  knowledge, this is the first paper on the joint optimization of periodic and continuous strategies for \lev processes. When there are jumps in the underlying process, the solution methods used in \cite{ATW2} for the Brownian motion model can no longer be applied. However, by expressing the expected NPVs via the scale function, a direct approach is possible without restricting the underlying \lev measure.  The same techniques are expected to be applied to other  stochastic control problems, such as inventory control \cite{Yamazaki2016} and dividend problems with capital injections \cite{APP2007, BKY},  when considering their extensions with periodic and continuous strategies.



In most optimal dividend problems,  a single barrier separates the waiting region from the controlling region. However, in this problem, three regions are separated by two barriers that are identified by the smooth fit conditions.  For this reason, the biggest challenges are to show the existence of a pair of barriers satisfying these conditions and to conduct the analysis on each of these three regions that are required for verification.  The former can be handled by observing that the desired barriers are such that a function of two variables and its partial derivative vanish simultaneously.  For the latter, although the shape of the value function differs in these three regions, finding a unifying expression is possible using the scale function. Moreover, the variational inequalities can be efficiently proven using this
unifying expression.

The rest of this paper is organized as follows. Section  \ref{dividends-strategy} reviews the spectrally positive \lev process and formulates the problem considered in this paper. 
Section \ref{section_model} then  defines the hybrid barrier strategy and constructs the corresponding controlled surplus process. Section \ref{section_verificaiton_lemma} obtains the verification lemma (sufficient conditions for optimality). 
In Section \ref{section_computation_NPV}, we review scale functions and compute the expected NPV of the dividends under the hybrid barrier strategy.  
Sections \ref{section_smooth_fit} and \ref{section_optimality} provide solutions for the case wherein the hybrid barrier strategy is optimal: we select the candidate optimal barriers in Section \ref{section_smooth_fit} and show that the corresponding candidate value function solves the required variational inequalities in Section \ref{section_optimality}.
Section \ref{section_other_cases} considers the case where pure continuous/periodic barrier strategies are optimal.  Finally, we conclude the paper by presenting some numerical results in Section \ref{section_numerics}.  
Some proofs are deferred to the appendix.

	Throughout the paper,  $f(x+) := \lim_{y \downarrow x}f(y)$ and $f(x-)  := \lim_{y \uparrow x}f(y)$ are used to indicate the right- and left-hand limits, respectively, for any function $f$ whenever they exist.
We also let 
$\Delta \zeta(s):= \zeta(s)-\zeta(s-)$ and $\Delta w(\zeta(s)):=w(\zeta(s))-w(\zeta(s-))$ for any process $\zeta$ with left-hand limits. 
	
	\section{Problem formulation}\label{dividends-strategy} 
	
	
	\subsection{Spectrally positive L\'evy processes} Let $X=(X(t); t\geq 0)$ be a L\'evy process defined on a  probability space $(\Omega, \mathcal{F}, \p)$.  For $x\in \R$, we denote by $\p_x$ the law of $X$ when it starts at $x$, and, for convenience, we write  $\p$ in place of $\p_0$. Accordingly, we  write the associated expectation operators as  $\e_x$ and $\e$. In this paper, we  assume that $X$ is \textit{spectrally positive},   which means that it has no negative jumps and is not a subordinator.  We also assume throughout this work that its Laplace exponent $\psi:[0,\infty) \to \R$, i.e.,
\[
\e\big[{\rm e}^{-\theta X(t)}\big]=:{\rm e}^{\psi(\theta)t}, \qquad t, \theta\ge 0,
\] 
is given by the \emph{L\'evy-Khintchine formula}
\begin{equation}\label{lk}
\psi(\theta):=\gamma\theta+\frac{\sigma^2}{2}\theta^2+\int_{(0, \infty)}\big({\rm e}^{-\theta z}-1+\theta z\mathbf{1}_{\{z<1\}}\big)\Pi(\ud z), \quad \theta \geq 0,
\end{equation}
where $\gamma \in \R$, $\sigma\ge 0$, and $\Pi$ is a measure on $(0, \infty)$, known as the L\'evy measure of $X$, which satisfies
$\int_{(0, \infty)}(1\land z^2)\Pi(\ud z)<\infty$.


It is well known that $X$ has paths of bounded variation if and only if $\sigma=0$ and $\int_{(0,1)} z\Pi(\mathrm{d} z) < \infty$; in this case, $X$ can be written as
\begin{equation}
X(t)=-ct+S(t), \,\,\qquad t\geq 0,\notag
\end{equation}
where 
\begin{align}
	c:=\gamma+\int_{(0,1)} z\Pi(\mathrm{d}z) \label{def_drift_finite_var}
\end{align}
and $(S(t); t\geq0)$ is a driftless subordinator. Note that  $c>0$, since we have ruled out the case where $X$ has monotone paths. In this case, its Laplace exponent is given by
\begin{equation*}
	\psi(\theta) = c \theta+\int_{(0,\infty)}\big( {\rm e}^{-\theta z}-1\big)\Pi(\ud z), \quad \theta \geq 0.
\end{equation*}
As is commonly assumed in the literature, we assume that $X$ drifts to infinity, i.e.,
\begin{align} \label{assumption_drift}
	\E [X_1] = - \psi'(0+) \in (0, \infty).
\end{align}


\subsection{Problem}
We assume that dividend payments are made both periodically and continuously.

The set of periodic dividend decision times is denoted as $\mathcal{T}_r :=(T(i); i\geq 1 )$, where $T(i)$, for each $i\geq  1$, represents the $i^{\textrm{th}}$ arrival time of the Poisson process $N^r=( N^r(t); t\geq 0) $ with intensity $r>0$, which is independent of the L\'evy process $X$. 
Let $\mathbb{F} := (\mathcal{F}(t); t \geq 0)$ be the filtration generated by the process $(X, N^r)$.

\par A strategy  $\pi := \left( L^{\pi}(t); t \geq 0 \right)$ is a nondecreasing, right-continuous, and $\mathbb{F}$-adapted process, where  the cumulative dividend amount $L^{\pi}$ has the following decomposition:
\[
L^{\pi}(t):= L^\pi_c(t) + L^\pi_p(t),\qquad\text{$t\geq0$.}
\]
Here, $L^\pi_c(t)$ is a non-decreasing, right-continuous, and $\mathbb{F}$-adapted process with $L^\pi_c(0-)=0$ that models the aggregate dividend until $t$ for the continuous strategy.   In contrast, $L^\pi_p(t)$ is a non-decreasing and $\mathbb{F}$-adapted process of the form
\begin{align}
L^\pi_p(t)=\int_{[0,t]}\nu^{\pi}(s)\diff N^r(s),\qquad\text{$t\geq0$,} \label{L_p_form}
\end{align}
for some $\mathbb{F}$-adapted c\`agl\`ad process $\nu^{\pi}$.
We note from \eqref{L_p_form} that the periodic dividend payment at time $T(i)$  is given by $\nu^\pi(T(i))$ for all $i\geq 1$. 

The surplus process $U^\pi$ obtained after dividends are deducted is such that 
\[
U^\pi(t) := X(t)- L^\pi (t) =X(t)- L^\pi_c(t) - \sum_{i=1}^{\infty}\nu^{\pi}(T(i))1_{\{T(i)\leq t\}}, \qquad\text{$0\leq t\leq \eta_0^{\pi}$}, 
\]
where $\eta_0^\pi:=\inf\{t>0:U^{\pi}(t)<0\}$
 is the corresponding ruin time.  We let $\inf \varnothing = \infty$ at this instance and other instances throughout the paper. While the payment of dividends can cause immediate ruin by setting $\Delta L^{\pi}(t) = U^{\pi}(t-) + \Delta X(t)$, it cannot exceed the surplus amount that is currently available.  In other words, we assume that
\begin{align} \label{surplus_constraint}
0\leq  \Delta L^{\pi}(t) \leq U^{\pi}(t-) + \Delta X(t), \qquad\text{for $t \geq 0$.}
\end{align}
We let  $\mathcal{A}$ be the set of all admissible strategies  that satisfy all the constraints described above.

The problem is to maximize, for $q > 0$, the expected NPV of the dividends 
associated with the strategy $\pi\in \mathcal{A}$, defined as, for $x \geq 0$,
\begin{align} \label{v_pi}
\begin{split}
	v_{\pi} (x) &:= \mathbb{E}_x \Big( \int_{[0,\eta_0^{\pi}]} {\rm e}^{-q t} \big(\beta \diff L^\pi_c(t) + \diff L^\pi_p(t) \big)  \Big) = \mathbb{E}_x \Big(\int_{[0,\eta_0^{\pi}]} {\rm e}^{-q t} \beta \diff L^\pi_c(t) + \int_{[0,\eta_0^{\pi}]} {\rm e}^{-q t}  \nu^{\pi}(t)\diff N^r(t)  \Big), 
	\end{split}
\end{align}
where $\beta$ is the ratio between the net proportion of continuous and periodic dividends received per dollar (see Remark \ref{remark_beta}).
Hence the problem is to compute the value function
\begin{equation*}
	v(x):=\sup_{\pi \in \mathcal{A}}v_{\pi}(x), \quad x \geq 0,
\end{equation*}
and to obtain the optimal strategy $\pi^*$ that achieves it, if such a strategy exists.

\begin{remark} \label{remark_beta}In \cite{ATW2}, the objective is to maximize
\begin{align*}
	V_{\pi} (x) &:= \mathbb{E}_x \Big( \int_{[0,\eta_0^{\pi}]} {\rm e}^{-q t} \big(\beta_c \diff L^\pi_c(t) + \beta_p \diff L^\pi_p(t) \big)  \Big), \quad x \geq 0,
\end{align*}
where $\beta_c, \beta_p > 0$ are the net proportions of continuous and periodic dividends received per dollar, respectively.
Solving this is equivalent to maximizing \eqref{v_pi} by setting $\beta := \beta_c / \beta_p$. 
\end{remark}

\section{Hybrid barrier strategies} \label{section_model}


In this section, we define \emph{the hybrid 
	barrier strategy} $\pi_{a,b} = (L_p^{(a,b)}, L_c^{(a,b)})$ for $b \geq a \geq 0$.  
At each of the periodic dividend decision times $\mathcal{T}_r$ where the surplus is above $a$, the excess is paid (Parisian reflection).  In contrast, as with the classical barrier strategy, dividends are also paid such that the surplus never exceeds $b$ (classical reflection). The resulting controlled surplus process then becomes the following \emph{\lev process with Parisian and classical reflection above}: 
\begin{align}
U^{(a,b)}_r(t) = X(t) - L_p^{(a,b)} (t) - L_c^{(a,b)}(t), \quad t \geq 0, \label{decomposition_U_a_b}
\end{align}
where $L_p^{(a,b)}(t)$ and $L_c^{(a,b)}(t)$ are the respective amounts of Parisian and classical reflection (dividends under the periodic and continuous barrier strategies) accumulated up to time $t$.
The formal construction of 
 $(U^{(a,b)}_r, L_p^{(a,b)}, L_c^{(a,b)})$ is given in Section \ref{section_construction_parisian_classical}.

\subsection{\lev processes with Parisian reflection above} \label{subsection_process_defined}




Before constructing the process $U^{(a,b)}_r$ given in \eqref{decomposition_U_a_b}, we first review the \emph{spectrally positive \lev process with Parisian reflection above} (and without classical reflection above), which is denoted by $X^a_r = (X^a_r(t); t \geq 0)$ for a fixed (Parisian) barrier $a$. The dual of this process is exactly  the one studied in \cite{APY}.

This process is only observed at times $\mathcal{T}_r$ and is pushed down to $a$ if and only if it is above $a$.
More specifically, we have
\begin{align} \label{X_X_r_the_same}
X_r^a(t) = X(t), \quad 0 \leq t < T_a^+(1)
\end{align}
where
\begin{align} T_{a}^+(1) := \inf\{T(i):\; X(T(i)) > a\}. \label{def_T_0_1}
\end{align}
The process then jumps down by $X(T_a^+(1))-a$ so that $X_r^a(T_a^+(1)) =a$. For $T_a^+(1) \leq t < T_a^+(2)  := \inf\{T(i) > T_a^+(1):\; X_r^a(T(i)-) > a\}$, we have $X^a_r(t) = X(t) - (X(T_a^+(1))-a)$.  The process can be constructed by repeating this procedure.

Suppose $L_p^a(t)$ is the cumulative amount of (Parisian) reflection up to time $t \geq 0$. We then have
\begin{align*}
X_r^a(t) = X(t) - L^a_p(t), \quad t \geq 0,
\end{align*}
with
\begin{align}
L^a_p(t) := \sum_{T^+_a(i) \leq t} (X_r^a(T_a^+(i)-)-a), \quad t \geq 0, \label{def_L_r}
\end{align}
where $(T_{a}^+(n); n \geq 1)$ can be constructed inductively using \eqref{def_T_0_1} and
\begin{eqnarray*} T_{a}^+(n+1) := \inf\{T(i) > T_a^+(n):\; X_r^a(T(i)-) > a\}, \quad n \geq 1.
\end{eqnarray*}

\subsection{\lev processes with Parisian and classical reflection above} \label{section_construction_parisian_classical}
We now construct the process $U_r^{(a,b)}$ by adding classical reflection in $X_r^a$. For $b \geq a \geq 0$, let
\begin{align}
U^b(t) := X(t) - L^b_c(t) \quad \textrm{where } L^b_c(t) := \sup_{0 \leq s \leq t} (X(s)-b ) \vee 0, \quad t \geq 0, \label{classical_reflected}
\end{align}
be the process classically reflected from above at $b$ (see Bayraktar et al.\ \cite{BKY}).  \par We  have
\begin{align}
U^{(a,b)}_r(t) = U^b(t), \quad 0 \leq t < \widehat{T}_a^{+} (1) \label{Y_matches}
\end{align}
where $\widehat{T}_{a}^+(1) := \inf\{T(i):\; U^{b}(T(i)) > a\}$.
The process then jumps down by $U^{b}(\widehat{T}_0^+(1))-a$ so that $U^{(a,b)}_r(\widehat{T}_a^+(1)) = a$. For $\widehat{T}_a^+(1) \leq t < \widehat{T}_a^+(2)  := \inf\{T(i) > \widehat{T}_a^+(1):\; U^{(a,b)}_r(T(i) -) > a\}$, the process $U_r^{(a,b)}(t)$ is the classical reflected process (in the form \eqref{classical_reflected}) of the shifted process $X(t) - X(\widehat{T}_a^+(1))+a$. 
The process can be constructed by repeating this procedure. 
This process admits the decomposition \eqref{decomposition_U_a_b} because the cumulative periodic dividend process becomes \eqref{def_L_r} when $X^a_r$ is replaced with $U^{(a,b)}_r$, whereas, between periodic payments, the process behaves like the classical reflected process \eqref{classical_reflected} with different starting points.

\subsection{Hybrid barrier strategies and their special cases} \label{subsection_various_strategies} 
For $0 \leq a \leq b$, the hybrid barrier strategy $\pi_{a,b} = (L_p^{(a,b)}, L_c^{(a,b)})$ is an admissible strategy, with  $L_p^{(a,b)}$ taking the form \eqref{L_p_form}.

The \emph{pure periodic barrier strategy} $\pi_{a, \infty}$ corresponds to the case $b = \infty$.  This strategy is clearly admissible and the resulting aggregate periodic dividend and controlled surplus processes become $L_p^a$ and $X_r^a$, respectively, as described in Section \ref{subsection_process_defined}.  The aggregate continuous dividend process
is uniformly zero.

The \emph{pure continuous barrier strategy} $\pi_{b,b}$ corresponds to the case $a = b$.  It is an admissible strategy, and the resulting aggregate continuous dividend and controlled surplus processes are $L_c^b$ and $U^b$, respectively, as in \eqref{classical_reflected}. The aggregate periodic dividend process
is uniformly zero.

Finally, if $0 = a < b$, liquidation occurs
at the first periodic dividend payment opportunity.

\section{Sufficient conditions for optimality} \label{section_verificaiton_lemma}


The aim of this paper is to to show the optimality of a hybrid  barrier strategy as defined in Section \ref{section_model}.  Toward this end, we first derive the \emph{verification lemma} and obtain sufficient conditions for optimality. This is a generalization of Lemma 4.1 of \cite{ATW2}. For verification of related stochastic control problems driven by spectrally one-sided \lev processes, see Section 5.4 of \cite{APP2007}, Lemma 4.1 of \cite{HPY}, and  the proof of Lemma 6.1 in \cite{Yamazaki2016}. 

We call a measurable function $g$ \emph{sufficiently smooth} if $g$ is $C^1 (0,\infty)$ (resp.\ $C^2 (0,\infty)$) when $X$ has paths of bounded (resp.\ unbounded) variation.
We let $\mathcal{L}$ be the operator for $X$ acting on a sufficiently smooth function $g$, defined by
\begin{equation}
\begin{split}
\mathcal{L} g(x)&:= -\gamma g'(x)+\frac{\sigma^2}{2}g''(x) +\int_{(0,\infty)}[g(x + z)-g(x)-g'(x)z\mathbf{1}_{\{0<z<1\}}]\Pi(\mathrm{d}z). \label{generator}
\end{split}
\end{equation}
\begin{lemma}[Verification lemma]
	\label{verificationlemma_2}
	Suppose $\hat{\pi} \in \mathcal{A}$ is such that $v_{\hat{\pi}}$ is sufficiently smooth on $(0,\infty)$
	, and satisfies
	\begin{align}
		\label{HJB-inequality_db}
		(\mathcal{L} - q)v_{\hat{\pi}}(x)+r\max_{0\leq l\leq x}\{l+v_{\hat{\pi}}(x-l)-v_{\hat{\pi}}(x)\}\leq 0,  &\quad x > 0,   \\
		v_{\hat{\pi}}'(x)\geq \beta,  &\quad x > 0.  \label{HJB-inequality_db2} 
	\end{align}
	Then $v_{\hat{\pi}}(x)=v(x)$ for all $x\geq0$ and hence $\hat{\pi}$ is an optimal strategy.
\end{lemma}
\begin{proof}
By the definition of $v$ as a supremum, it follows that $v_{\hat{\pi}}(x)\leq v(x)$ for all $x\geq0$. We write $w:=v_{\hat{\pi}}$ and show that $w(x)\geq v_\pi(x)$ for all $\pi\in\mathcal{A}$ and $x\geq0$. 

Fix $x \geq 0$, $\pi\in \mathcal{A}$, and the corresponding surplus process $U^\pi$. 
Let $(T_n)_{n\in\mathbb{N}}$ be the sequence of stopping times defined by $T_n :=\inf\{t>0:U^\pi (t) > n \text{ or }  U^\pi(t)<1/n\}$. 
Since $U^\pi$ is a semi-martingale and $w$ is sufficiently smooth on $(0, \infty)$ by assumption, we can use the change of variables for the bounded variation case (Theorem II.31 of \cite{protter}) and It\^o's formula (Theorem II.32 of \cite{protter}) for the unbounded variation case
 to the stopped process $({\rm e}^{-q(t\wedge T_n)}w(U^\pi(t\wedge T_n)); t \geq 0)$. Because $X$ and $N^r$ do not jump simultaneously, we can write under $\mathbb{P}_x$ that  
	\begin{equation*}
		\begin{split}
			{\rm e}^{-q(t\wedge T_n)}&w(U^{\pi}(t\wedge T_n))  -w(x)
			\\
			= &  \int_{0}^{t\wedge T_n}{\rm e}^{-qs}   (\mathcal{L}-q)w(U^{\pi}(s-))   \mathrm{d}s
			-\int_{[0,t\wedge T_n]}{\rm e}^{-qs}\nu^{\pi}(s)\mathrm{d}N^{r}(s)\\
			&- \int_0^{t\wedge T_n}{\rm e}^{-qs}w'(U^{\pi}(s-)) \mathrm{d} L_c^{\pi,c}(s) + \sum_{0 \leq s\leq t\wedge T_n}{\rm e}^{-qs}\Delta w \Big(U^{\pi}(s-)-\Delta L_c^{\pi}(s) \Big) 1_{\{ \Delta L_c^\pi (s) \neq 0 \}}\\
			&+\int_0^{t\wedge T_n}{\rm e}^{-qs}r\left\{\nu^{\pi}(s)+w \Big( U^{\pi}(s-)-\nu^{\pi}(s) \Big)-w(U^{\pi}(s-))\right\} \mathrm{d}s  + M(t \wedge T_n),
		\end{split}
	\end{equation*}
where $L_c^{\pi,c}$ is the continuous part of $L_c^{\pi}$ and $M$ is a local zero-mean martingale as in (A.1) of \cite{YP2016}. 

	\par On the other hand,  using \eqref{HJB-inequality_db2},
	we obtain that
	\begin{multline*}
		-\int_0^{t\wedge T_n}{\rm e}^{-qs}w'(U^{\pi}(s-)) \mathrm{d} L_c^{\pi,c}(s)+\sum_{0 \leq s\leq t\wedge T_n}{\rm e}^{-qs}[\Delta w(U^{\pi}(s-)-\Delta L_c^{\pi}(s))] 1_{\{ \Delta L_c^\pi (s) \neq 0 \}}\\
		\leq -\beta \int_0^{t\wedge T_n} {\rm e}^{-qs} \mathrm{d} L_c^{\pi,c}(s)-\beta\sum_{0 \leq s\leq t\wedge T_n} {\rm e}^{-qs}\Delta L^{\pi}_c(s)=-\beta\int_{[0, t\wedge T_n]}{\rm e}^{-qs} \mathrm{d} L_c^{\pi}(s).
	\end{multline*}
	\par 
By this inequality,  \eqref{surplus_constraint}  and because ${\rm e}^{-q(t\wedge T_n)}w(U^{\pi}(t\wedge T_n))$ is nonnegative,
	\begin{equation*}
		\begin{split}
			w(x) &\geq 
			-\int_{0}^{t\wedge T_n}{\rm e}^{-qs}  \Big[ (\mathcal{L}-q)w(U^{\pi}(s-))+r\max_{0\leq l \leq U^{\pi}(s-)}\left\{ l+w(U^{\pi}(s-)-l)-w(U^{\pi}(s-))\right\} \Big]  \mathrm{d}s\\
			& +\int_{[0, t\wedge T_n]}{\rm e}^{-qs}\nu^{\pi}(s)\mathrm{d}N^{r}(s)+\beta\int_{[0, t\wedge T_n]}{\rm e}^{-qs} \mathrm{d} L^{\pi}_c(s) - M(t\wedge T_n).
		\end{split}
	\end{equation*}
	Using \eqref{HJB-inequality_db}, 
	we have
	\begin{equation} \label{w_lower_2}
	\begin{split}
	w(x) \geq &
	- M(t\wedge T_n)+ \int_{[0, t\wedge T_n]}{\rm e}^{-qs}\nu^{\pi}(s)\mathrm{d}N^{r}(s)+\beta\int_{[0, t\wedge T_n]}{\rm e}^{-qs} \mathrm{d} L_c^{\pi}(s). 
	\end{split}
	\end{equation} 
	\par Now  taking expectations  in \eqref{w_lower_2} 
and
	letting $t$ and $n$ go to infinity ($T_n\nearrow\infty$ $\mathbb{P}_x$-a.s.), 
	the monotone convergence theorem gives  
$w(x) \geq \lim_{t, n \rightarrow \infty}\mathbb{E}_x \left( \int_{[0,t\wedge T_n]}{\rm e}^{-qs}\nu^{\pi}(s)\mathrm{d}N^{r}(s)+\beta\int_{[0, t\wedge T_n]}{\rm e}^{-qs} \mathrm{d} L_c^{\pi}(s) \right) =v_{\pi}(x)$.
	This completes the proof.
\end{proof}

\section{Computation of the expected NPV of dividends under $\pi_{a,b}$} \label{section_computation_NPV}
%
In this section, we compute the expected NPV of the dividends under the hybrid barrier strategy $\pi_{a,b}$ as defined in Section \ref{section_model}:
\begin{align} \label{v_a_b_def}
	v_{a,b}(x) :=v_{\pi_{a,b}} (x) = \mathbb{E}_x \Big( \int_{[0,\eta_0^{(a,b)}]} {\rm e}^{-q t} [\beta \diff  L^{(a,b)}_c(t) + \diff L^{(a,b)}_p(t)] 
	\Big), \quad x \geq 0,
\end{align}
where $\eta_0^{(a,b)} := \inf \{ t > 0: U_r^{(a,b)} (t) < 0 \}$.
The pure periodic case ($b = \infty$) and the pure continuous case ($a = b$) are given in \cite{YP2016} and  \cite{BKY}, respectively.  Here, we focus on the case $0 \leq a < b$.

Toward this end, we first review the scale function.

	\subsection{Review on scale functions.}
	
	Fix $q > 0$. 
	We use $W^{(q)}$ for the scale function of the process $X$. This is the mapping from $\R$ to $[0, \infty)$ that takes value zero on the negative half-line, while on the positive half-line it is a strictly increasing function that is defined by its Laplace transform:
	\begin{align} \label{scale_function_laplace}
	\begin{split}
	\int_0^\infty  {\rm e}^{-\theta x} W^{(q)}(x) \diff x &= \frac 1 {\psi(\theta)-q}, \quad \theta > \Phi(q),
	\end{split}
	\end{align}
	where $\psi$ is as defined in \eqref{lk} and
	\begin{align}
	\begin{split}
	\Phi(q) := \sup \{ \lambda \geq 0: \psi(\lambda) = q\} . 
	\end{split}
	\label{def_varphi}
	\end{align}
	We also define, for $x \in \R$, 
	\begin{align*}
	\overline{W}^{(q)}(x) :=  \int_0^x W^{(q)}(y) \diff y, \quad
	Z^{(q)}(x) := 1 + q \overline{W}^{(q)}(x), \quad \textrm{and}  \quad
	\overline{Z}^{(q)}(x) := \int_0^x Z^{(q)} (z) \diff z. 
	\end{align*}
	
Define also
\begin{align*} 
Z^{(q)}(x, \theta ) &:={\rm e}^{\theta x} \left( 1 + (q- \psi(\theta )) \int_0^{x} {\rm e}^{-\theta  z} W^{(q)}(z) \diff z	\right), \quad x \in \R, \, \theta  \geq 0.
\end{align*}
In particular, $Z^{(q)}(x, 0) =Z^{(q)}(x)$, $x \in \R$, and, for $r > 0$,
\begin{align*} 
\begin{split}
Z^{(q)}(x, \Phi(q+r)) &:={\rm e}^{\Phi(q+r) x} \left( 1 -r \int_0^{x} {\rm e}^{-\Phi(q+r) z} W^{(q)}(z) \diff z \right).
\end{split}
\end{align*}

	
	If we define 
	 $\tau_a^- := \inf \left\{ t \geq 0: X(t) < a \right\}$ and $\tau_b^+ := \inf \left\{ t \geq 0: X(t) >  b \right\}$ for any $b > a$,
	then, for $x \geq a $, 
	\begin{align}
	\begin{split}
	\E_x \left( {\rm e}^{-q \tau_a^-} 1_{\left\{ \tau_b^+ > \tau_a^- \right\}}\right) &= \frac {W^{(q)}(b-x)}  {W^{(q)}(b-a)}, \\
	\E_x\left({\rm e}^{-q\tau_b^{+}+\theta(b-X(\tau_b^{+}))};\tau_b^+< \tau_a^{-} \right)&=Z^{(q)}(b-x, \theta)-W^{(q)}(b-x)\frac{Z^{(q)}(b-a,\theta)}{W^{(q)}(b-a)}, \quad \theta \geq 0;
	\end{split}
	\label{laplace_in_terms_of_z}
	\end{align}
	see (2.10) of \cite{YP2016b} for the spectrally negative case.

	\begin{remark} \label{remark_smoothness_zero}
		\begin{enumerate}
			\item If $X$ is of unbounded variation or the \lev measure is atomless, it is known that $W^{(q)}$ is $C^1(\R \backslash \{0\})$; see, e.g.,\ \cite[Theorem 3]{Chan2011}.
			\item Regarding the asymptotic behavior near zero, as in Lemmas 3.1 and 3.2 of \cite{KKR},
			\begin{align}\label{eq:Wqp0}
			\begin{split}
			W^{(q)} (0) &= \left\{ \begin{array}{ll} 0 & \textrm{if $X$ is of unbounded
				variation,} \\ \frac 1 {c} & \textrm{if $X$ is of bounded variation,}
			\end{array} \right. \\
			W^{(q)\prime} (0+) &=
			\left\{ \begin{array}{ll}  \frac 2 {\sigma^2} & \textrm{if }\sigma > 0, \\
			\infty & \textrm{if }\sigma = 0 \; \textrm{and} \; \Pi(0,\infty)= \infty, \\
			\frac {q + \Pi(0,\infty)} {c^2} &  \textrm{if }\sigma = 0 \; \textrm{and} \; \Pi(0,\infty) < \infty.
			\end{array} \right.
			\end{split}
			\end{align}
			On the other hand, as in Lemma 3.3 of \cite{KKR},
			\begin{align}
			\begin{split}
			{\rm e}^{-\Phi(q) x}W^{(q)} (x) \nearrow \psi'(\Phi(q))^{-1}, \quad \textrm{as } x \uparrow \infty.
			\end{split}
			\label{W^{(q)}_limit}
			\end{align}
			


		\end{enumerate}
	\end{remark}
	
In this paper, we use the following versions of the scale function: for $a < b$ and $x \in \R$,  
\begin{align} \label{W_a_def}
	\begin{split}
		W_{a-b}^{(q,r)}(a-x) &:=W^{(q+r)}(b-x)- r\int_{0}^{a-x}  W^{(q)}(a-x-y)W^{(q+r)}(y-a+b)\diff y, \\ 
		\overline{W}_{a-b}^{(q,r)}(a-x) &:=\overline{W}^{(q+r)}(b-x)- r\int_{0}^{a-x}  W^{(q)}(a-x-y)\overline{W}^{(q+r)}(y-a+b)\diff y, \\ 
		Z_{a-b}^{(q,r)}(a-x, \theta) 
		&:=Z^{(q+r)}(b-x, \theta)- r\int_{0}^{a-x}  W^{(q)}(a-x-y)Z^{(q+r)}(y-a+b, \theta)\diff y, \quad \theta \geq 0, \\ 
		\overline{Z}_{a-b}^{(q,r)}(a-x) &:=\overline{Z}^{(q+r)}(b-x)- r\int_{0}^{
			a-x}  W^{(q)}(a-x-y) \overline{Z}^{(q+r)}(y-a+b)\diff y.
	\end{split}
\end{align}

\begin{remark} \label{remark_Z_q_r_another_expression}
For $a < b$ and $x \in \R$, we can write 
\begin{align*} 
		W_{a-b}^{(q,r)}(a-x) 
		&=W^{(q)}(b-x)+r\int_0^{b-a}W^{(q)}(b-u-x)W^{(q+r)}(u) \diff u, \\
		Z_{a-b}^{(q,r)}(a-x, \theta) 
		&=Z^{(q)}(b-x, \theta)+r\int_0^{b-a}W^{(q)}(b-u-x)Z^{(q+r)}(u, \theta) \diff u, \quad \theta \geq 0,\\
		\overline{Z}_{a-b}^{(q,r)}(a-x) 
		 &=\overline{Z}^{(q)}(b-x)+r\int_0^{b-a}W^{(q)}(b-u-x)\overline{Z}^{(q+r)}(u) \diff u.
\end{align*}
These hold by (5) of \cite{LRZ} and (3.4) of \cite{YP2015}. 
\end{remark}
	


	
\subsection{The expression for $v_{a,b}$ in terms of the scale function} 
\par 

We will now present some preliminary results that will allow us to express the function $v_{a,b}$ given by \eqref{v_a_b_def} in terms of scale functions. 
The proofs of the following results are lengthy and hence are deferred to Appendix \ref{appendix_proof}.

Define 
\begin{align}
	\mathcal{K}^{(q,r)}_{a-b}(a-x):=\frac{1}{r+q}\left(q\frac{Z_{a-b}^{(q,r)}(a-x)}{Z^{(q+r)}(b-a)}+rZ^{(q)}(a-x)\right), \quad a < b, x \in \R. \label{def_K}
\end{align}

	\begin{proposition}\label{per-con}
	For $q>0$, $0 \leq a<b$, and $x \geq 0$,
	\begin{align}
	\E_x\left(\int_{[0, \eta_0^{(a,b)}]} {\rm e}^{-qt} \diff L_p^{(a,b)}(t)\right) =\frac{\mathcal{K}^{(q,r)}_{a-b}(a-x)}{\mathcal{K}^{(q,r)}_{a-b}(a)}k^{(q,r)}_{a-b}(a)-k^{(q,r)}_{a-b}(a-x), \label{L_p_identity}
	\end{align}
	where
	\begin{align*}
	k^{(q,r)}_{a-b}(a-x) := \frac r {q+r}\left(\overline{Z}^{(q)}(a-x)- \frac r q \overline{Z}^{(q+r)}(b-a)  { Z^{(q)}(a-x)}  -\overline{Z}^{(q,r)}_{a-b}(a-x)\right).
	\end{align*}
	\end{proposition}
	\begin{proposition}\label{sin-con}
	For $q>0$, $0 \leq a<b$, and $x \geq 0$,
\begin{align} \label{L_c_identity}
	\E_x\left(\int_{[0, \eta_0^{(a,b)}]}{\rm e}^{-qt}\diff  L_c^{(a,b)}(t)\right) =\frac{\mathcal{K}^{(q,r)}_{a-b}(a-x)}{\mathcal{K}^{(q,r)}_{a-b}(a)} i^{(q,r)}_{a-b}(a)- i^{(q,r)}_{a-b}(a-x),
	\end{align}
where
		\begin{multline*}
	i_{a-b}^{(q,r)}(a-x)
	:=\overline{Z}^{(q,r)}_{a-b}(a-x)-\psi'(0+)\overline{W}^{(q,r)}_{a-b}(a-x) \\+\frac r q  Z^{(q)}(a-x) \left(\overline{Z}^{(q+r)}(b-a)-\psi'(0+)\overline{W}^{(q+r)}(b-a)\right).
	\end{multline*}
\end{proposition}
Let, for $0 \leq a<b$ and $x \geq 0$,
\begin{align*}
\Gamma(a,b;x) &:= \frac r {q+r} \overline{Z}^{(q)}(a-x)+ \frac {\beta \psi'(0+)} {q}  + \Big(\beta - \frac r {q+r}  \Big) \Big[ \overline{Z}^{(q,r)}_{a-b}(a-x) + \frac r q  Z^{(q)}(a-x)    \overline{Z}^{(q+r)}(b-a) \Big], 
\end{align*}
where, in particular,
\begin{align}\label{gamma_x0}
\begin{split}
\Gamma(a,b)&:= \Gamma(a,b;0) 
= \frac r {q+r} \overline{Z}^{(q)}(a)+ \frac {\beta \psi'(0+)} {q}   + \Big(\beta - \frac r {q+r}  \Big)  \Big[ \overline{Z}^{(q,r)}_{a-b}(a) + \frac r q  Z^{(q)}(a)    \overline{Z}^{(q+r)}(b-a) \Big]. 
\end{split}
\end{align}

Combining Propositions \ref{per-con} and \ref{sin-con} and after simplification, we have the following result. The proof is given in Appendix \ref{appendix_proof}.
\begin{theorem}\label{theorem_vf} 
	For $0 \leq a < b$ and  $x \geq 0$,
\begin{align}\label{vf_ff}
v_{a,b}(x)
&=\frac{\mathcal{K}^{(q,r)}_{a-b}(a-x)}{\mathcal{K}^{(q,r)}_{a-b}(a)}\Gamma(a,b)
-\Gamma(a,b;x),
\end{align}
where, in particular, for $x\geq b$, we have $v_{a,b}(x)=\beta(x-b)+v_{a,b}(b)$.
\end{theorem}
\section{The selection of $a^*$ and $b^*$} \label{section_smooth_fit}

For the optimality of the hybrid barrier strategy $\pi_{a,b}$ constructed in Section \ref{section_model}, we shall assume the following relation on the parameters $\beta$, $r$, and $q$. 
\begin{assump} \label{assumption_beta_r_q} 
	We assume that $r / (q+r) <  \beta <  1$.
\end{assump}
This is assumed throughout this section and also in Section \ref{section_optimality}. The cases $\beta \geq 1$ and $\beta \leq r / (q+r)$, where the pure continuous and periodic barrier strategies, respectively, are shown to be  optimal, are deferred to Section \ref{section_other_cases}.

\subsection{Smoothness conditions.}
In order to obtain the optimal thresholds, which we call  $a^*$ and $b^*$, we will ask that the value function $v_{a^*,b^*}$ is smooth enough. 
For any choice of $0 \leq a < b$, it is clear that $v_{a,b}$ is continuous for $x \geq 0$. Its derivatives, whenever they exist, are given by \begin{align*}
v_{a,b}^{(n)}(x)
&=\frac{\displaystyle \frac {\partial^n} {\partial x^n} \mathcal{K}^{(q,r)}_{a-b}(a-x)}{\mathcal{K}^{(q,r)}_{a-b}(a)} \Gamma(a,b)
- \frac {\partial^n} {\partial x^n}\Gamma(a,b;x), \quad n \geq 1.
\end{align*}

Fix $0 \leq a < b$.
We shall first compute, in the following lemma, the derivatives of $\mathcal{K}^{(q,r)}_{a-b}$ and $\Gamma$.  
Its proof is given in Appendix \ref{proof_lemma_derivatives}. 

\begin{lemma} \label{lemma_derivatives_Gamma_K}
Fix $0 \leq a < b$.
(i) For $x \neq a$, 
\begin{align} \label{H_derivatives}
\begin{split}
	\frac \partial {\partial x}\mathcal{K}^{(q,r)}_{a-b}(a-x)
	&=\frac{q}{Z^{(q+r)}(b-a)} \Big[ -W^{(q+r)}(b-x) \\ &+ r \Big(  W^{(q+r)}(b-a) \overline{W}^{(q)} (a-x) + \int_{0}^{a-x}  \overline{W}^{(q)}(a-x-y) W^{(q+r) \prime}(y-a+b)\diff y \Big)
	\Big], \\
	\frac {\partial^2} {\partial x^2}\mathcal{K}^{(q,r)}_{a-b}(a-x)
	&=\frac{q}{Z^{(q+r)}(b-a)} \Big[ W^{(q+r) \prime}(b-x) \\ &- r \Big(  W^{(q+r)}(b-a) W^{(q)} (a-x) +  \int_{0}^{a-x}  W^{(q)}(a-x-y) W^{(q+r) \prime}(y-a+b)\diff y \Big)
	\Big], \\
	\frac {\partial^3} {\partial x^3}\mathcal{K}^{(q,r)}_{a-b}(a-x)
	&=\frac{q}{Z^{(q+r)}(b-a)} \Big[ - W^{(q+r) \prime \prime}(b-x) + r \Big(  W^{(q+r)}(b-a) W^{(q) \prime} (a-x) \\ &+W^{(q)}(0) W^{(q+r) \prime}(b-x) + \int_{0}^{a-x}  W^{(q) \prime}(a-x-y) W^{(q+r) \prime}(y-a+b)\diff y \Big)
	\Big],
	\end{split}
\end{align}
where we assume that $W^{(q+r) \prime}(b-x)$ and $W^{(q+r) \prime \prime}(b-x)$ exist for the second and third equalities, respectively.



(ii) For $x \neq a$, 
\begin{align} \label{Gamma_derivatives}
\begin{split}
\frac \partial {\partial x}\Gamma(a,b;x) 
 &= -\frac r {q+r} Z^{(q)}(a-x)   - \Big(\beta - \frac r {q+r}  \Big) \Big[ Z^{(q+r)}(b-x) - r \Big(   Z^{(q+r)}(b-a) \overline{W}^{(q)} (a-x) \\ &+ (q+r) \int_{0}^{a-x}  \overline{W}^{(q)}(a-x-y) W^{(q+r)}(y-a+b)\diff y \Big) \Big], \\
\frac {\partial ^2} {\partial x^2}\Gamma(a,b;x) 
 &= \frac {qr} {q+r} W^{(q)}(a-x)   - \Big(\beta - \frac r {q+r}  \Big) \Big[ - (q+r)W^{(q+r)}(b-x) + r \Big(   Z^{(q+r)}(b-a) W^{(q)} (a-x) \\ &+ (q+r) \int_{0}^{a-x}  W^{(q)}(a-x-y) W^{(q+r)}(y-a+b)\diff y \Big)  \Big], \\
\frac {\partial ^3} {\partial x^3}\Gamma(a,b;x) 
 &= -\frac {qr} {q+r} W^{(q)\prime}(a-x)   - \Big(\beta - \frac r {q+r}  \Big) \Big[  (q+r)W^{(q+r)\prime}(b-x) - r \Big(  Z^{(q+r)}(b-a) W^{(q) \prime} (a-x) \\ &+ (q+r) \Big[   W^{(q)}(0) W^{(q+r)}(b-x) + \int_{0}^{a-x}  W^{(q) \prime}(a-x-y) W^{(q+r)}(y-a+b)\diff y \Big] \Big)  \Big],
 \end{split}
\end{align}
where we assume in the third equality that  $W^{(q)\prime}(a-x)$ and $W^{(q)\prime}(b-x)$ exist.
\end{lemma}

In view of the first and second equalities in \eqref{H_derivatives} and \eqref{Gamma_derivatives} and Remark \ref{remark_smoothness_zero}, $v_{a,b}$ is $C^1 ((0,b) \cup (b, \infty))$ (resp.\ $C^2 ((0,b) \cup (b, \infty))$) when $X$ has paths of bounded (resp.\ unbounded) variation.   
We now analyze the smoothness at the barriers $a$ and $b$.

\subsubsection{Smoothness at $b$}

By the first equalities in \eqref{H_derivatives} and \eqref{Gamma_derivatives}, 
\begin{align*}
v_{a,b}'(b-)=-q\frac{\Gamma(a,b)}{\mathcal{K}^{(q,r)}_{a-b}(a)}\frac{W^{(q+r)}(0)}{Z^{(q+r)}(b-a)}+\beta=-q\frac{\Gamma(a,b)}{\mathcal{K}^{(q,r)}_{a-b}(a)}\frac{W^{(q+r)}(0)}{Z^{(q+r)}(b-a)}+v_{a,b}'(b+).
\end{align*}
Hence, in view of Remark \ref{remark_smoothness_zero} (2), for the case of bounded variation, $v_{a,b}$ is continuously differentiable if 
\begin{equation}\label{cond_1}
\mathbf{C_b}: \Gamma(a,b)=0,
\end{equation}
while it holds automatically for the case of unbounded variation.

For the case of unbounded variation, again by Remark \ref{remark_smoothness_zero} (2) and  the second equalities in \eqref{H_derivatives} and \eqref{Gamma_derivatives}, 
\begin{align*}
v_{a,b}''(b-)=q\frac{\Gamma(a,b)}{\mathcal{K}^{(q,r)}_{a-b}(a)}\frac{W^{(q+r)\prime}(0+)}{Z^{(q+r)}(b-a)}=q\frac{\Gamma(a,b)}{\mathcal{K}^{(q,r)}_{a-b}(a)}\frac{W^{(q+r)\prime}(0+)}{Z^{(q+r)}(b-a)}+v_{a,b}''(b+).
\end{align*}
Hence, $v_{a,b}$ is twice continuously differentiable at $b$ if \eqref{cond_1} holds.


\subsubsection{Smoothness at $a$}


For the case of bounded variation, suppose $W^{(q+r) \prime}(b-a)$ exists.  By the second equalities in \eqref{H_derivatives} and \eqref{Gamma_derivatives}, 
\begin{align}\label{boundary_a_1}
	v_{a,b}''(a-)=v_{a,b}''(a+)&-  r W^{(q)}(0)\Big[ q\frac{W^{(q+r)}(b-a)}{Z^{(q+r)}(b-a)}\frac{\Gamma(a,b)}{\mathcal{K}^{(q,r)}_{a-b}(a)} +\Big(\frac{q}{r+q}  -  \Big(\beta - \frac r {q+r}  \Big)Z^{(q+r)}(b-a)\Big) \Big].
\end{align}
On the other hand, by straightforward differentiation (using Remark \ref{remark_Z_q_r_another_expression}), 
\begin{align}
\gamma(a,b) :=\frac \partial {\partial a}\Gamma(a,b)
= \frac r {q+r} Z^{(q)}(a) \kappa(b-a), \quad  0 < a < b, \label{gamma_small}
\end{align} 
where
\begin{align*}
\kappa(y) :=   1 - \frac 1 q    \Big(\beta (q+r) -r \Big)  Z^{(q+r)}(y), \quad y > 0.
\end{align*}

Hence $v_{a,b}$ is twice continuously differentiable at $a$ if \eqref{cond_1} holds and
\begin{align}\label{cond_2}
\mathbf{C_a:} \gamma(a,b)= 0 \Longleftrightarrow
	& \frac q {\beta (q+r)-r} =   Z^{(q+r)}(b-a)
\end{align}
is satisfied.

For the case of unbounded variation, suppose $W^{(q+r) \prime \prime}(b-a)$ exists.  By the third equalities in \eqref{H_derivatives} and \eqref{Gamma_derivatives} and Remark \ref{remark_smoothness_zero} (2), we obtain 
\begin{align}\label{boundary_a_2}
	v_{a,b}'''(a-)=v_{a,b}'''(a+)&+ r W^{(q)\prime}(0+) \Big[ q\frac{W^{(q+r)}(b-a)}{Z^{(q+r)}(b-a)}\frac{\Gamma(a,b)}{\mathcal{K}^{(q,r)}_{a-b}(a)} + \Big(\frac{q}{r+q}  - \Big(\beta - \frac r {q+r}  \Big) Z^{(q+r)}(b-a)\Big) \Big].
\end{align}	
Hence $v_{a,b}$ is thrice continuously differentiable at $a$ if \eqref{cond_1} and \eqref{cond_2} hold simultaneously.

In summary, we have the following. 

\begin{lemma} \label{lemma_smooth_fit_cond} Suppose  $0 \leq a < b$ are chosen so that $\mathbf{C_b}$ is satisfied.  Then, the following holds:
\begin{enumerate}
\item $v_{a,b}$ is $C^1$ (resp.\ $C^2$) on $(0, \infty)$ for the case of bounded (resp.\ unbounded) variation.
\item Suppose in addition that  $\mathbf{C_a}$ is satisfied. Then,
\begin{enumerate}
\item for the case of bounded variation and if $W^{(q+r) \prime}(b-a)$ exists, then $v_{a,b}$ is twice continuously differentiable at $a$;
\item  for the case of unbounded variation and if $W^{(q+r) \prime \prime}(b-a)$ exists, then $v_{a,b}$ is thrice continuously differentiable at $a$.  
\end{enumerate}
\end{enumerate}
\end{lemma}
\begin{remark}
In view of Lemma \ref{lemma_smooth_fit_cond}, the smoothness obtained in (1) is already sufficient to apply the verification lemma (Lemma \ref{verificationlemma_2}). We therefore do not require the existence of  $W^{(q+r) \prime}(b-a)$ (resp.\ $W^{(q+r) \prime \prime}(b-a)$) for the case of bounded (resp.\ unbounded) variation as in (2).
\end{remark}
\subsection{Existence of $a^*, b^*$} \label{subsection_existence}

\par In view of \eqref{gamma_small}, by Assumption \ref{assumption_beta_r_q},
\begin{align*}
\frac \partial {\partial a}\kappa(b-a)
= \frac {q+r} q \Big( \beta (q+r) -r \Big)  W^{(q+r)}(b-a) > 0, \quad b > a > 0,
\end{align*}
and
\begin{align}\label{min_a}
\gamma(b-,b) 
= \frac r {q+r} Z^{(q)}(b)  \Big( 1 - \frac 1 q    \Big( \beta (q+r) - r \Big)  \Big) = \frac r {q} (1 - \beta) Z^{(q)}(b) >0, \quad b > 0.
\end{align}

Hence, in view of \eqref{gamma_small}, for $b > 0$, there exists a minimizer $a(b) \in [0, b)$ of the mapping $a \mapsto \Gamma(a, b)$ such that it is strictly decreasing on $[0,a(b))$ and strictly increasing on $(a(b), b)$.  In particular,
\begin{align}
a(b) = 0 \Longleftrightarrow
\gamma(0+,b)
 \geq 0
\Longleftrightarrow \kappa(b) \geq 0 \Longleftrightarrow b \leq \epsilon, \quad b > 0, \label{a_b_zero_equivalence}
\end{align}
where  $\epsilon > 0$ is such that 
\begin{align}\label{def-epsilon}
1 - \frac 1 q    \Big( \beta (q+r) - r\Big)  Z^{(q+r)}(\epsilon) = 0;
\end{align} 
this exists and is unique because, by Assumption \ref{assumption_beta_r_q}, the left hand side is strictly decreasing in $\epsilon$ to minus infinity and
\begin{align*}
1 - \frac 1 q    \Big( \beta (q+r) -r \Big)  Z^{(q+r)}(0) = \frac {q+r} q (1-\beta) > 0.
\end{align*} 
With this $\epsilon$, we have $a(b) = (b-\epsilon) \vee 0$ and
\begin{align}
\underline{\Gamma}(b) := \min_{0 \leq a \leq b} \Gamma(a,b)  = \Gamma((b-\epsilon) \vee 0,b). \label{Gamma_min_expression}
\end{align}
In summary, we have the following:
\begin{lemma} \label{lemma_Gama_wrt_a}Fix $b > 0$.
\begin{enumerate}
\item If $b \leq \epsilon$, the mapping $a \mapsto \Gamma(a, b)$ is monotonically (strictly) increasing and $\underline{\Gamma}(b) = \Gamma(0,b)$.
\item Otherwise, $a \mapsto \Gamma(a, b)$ decreases (strictly)  on  $[0,b-\epsilon)$ and increases (strictly)  on $(b-\epsilon, b)$. It attains a local minimum at $a(b) = b - \epsilon$ and hence $\gamma(b-\epsilon,b)
= 0$.
\end{enumerate}
\end{lemma}



We now show the following monotonicity property.
\begin{lemma} \label{Gamma_inf_increasing}
The function $\underline{\Gamma}(b)$ is (strictly) increasing to $\infty$.
\end{lemma}
\begin{proof}
By Remark \ref{remark_Z_q_r_another_expression}, for $b > a \geq 0$,
\begin{align*}
\frac \partial {\partial b}\Gamma(a,b)
&= \Big(\beta - \frac r {q+r}  \Big)  \Big[Z^{(q)} (b) + r W^{(q)}(a) \overline{Z}^{(q+r)}(b-a) \\ &+ r \int_0^{b-a} W^{(q)\prime} (b-u) \overline{Z}^{(q+r)} (u) \diff u + \frac r q  Z^{(q)}(a)  Z^{(q+r)}(b-a) \Big]. 
\end{align*}

By this and Assumption \ref{assumption_beta_r_q}, we obtain a uniform lower bound
\begin{align} \label{Gamma_derivative_b_uniform_bound}
\inf_{0 \leq a < b} \frac \partial {\partial b} \Gamma(a,b) & \geq \beta - \frac r {q+r}  > 0.
\end{align}

Recall $\epsilon > 0$ such that \eqref{def-epsilon} holds.
Fix $b_1 < b_2$ such that $b_2 - b_1 < \epsilon$ (or $b_2 - \epsilon < b_1$).
\begin{enumerate}
\item Suppose $b_2 \leq \epsilon$.  Then, because $b_1 - \epsilon, b_2 - \epsilon \leq 0$ and by \eqref{Gamma_min_expression} and \eqref{Gamma_derivative_b_uniform_bound},
\begin{align*}
\underline{\Gamma}(b_2) = \Gamma(0, b_2) > \Gamma(0, b_1) = \underline{\Gamma}(b_1).
\end{align*}
\item Suppose $b_2 > \epsilon$. Then, because $b_2 - \epsilon \in (0, b_1)$  and by \eqref{Gamma_min_expression} and \eqref{Gamma_derivative_b_uniform_bound},
\begin{align*}
\underline{\Gamma}(b_2) = \min_{0 \leq y \leq b_1} \Gamma (y, b_2) > \min_{0 \leq y \leq b_1} \Gamma (y, b_1) =  \underline{\Gamma}(b_1). \end{align*}
\end{enumerate}

Finally, it is clear by \eqref{Gamma_derivative_b_uniform_bound} that $\underline{\Gamma}(b) \xrightarrow{b \uparrow \infty} \infty$.
\end{proof}

In view of Lemma \ref{Gamma_inf_increasing}, we first start at $b=0$ and increase the value of $b$ until we attain the desired pair $(a^*, b^*)$ such that 
\begin{align*}
\underline{\Gamma}(b^*) = 0 \quad \textrm{and} \quad a^* = a(b^*).
\end{align*}

For the case $b = 0$, by \eqref{assumption_drift},
\begin{align}
\underline{\Gamma} (0) = \Gamma(0,0)  &=\beta\psi'(0+)/q  <0. \label{Gamma_zero_inf}
\end{align}


Hence we shall increase the value of $b$ until we get $b^*$ such that $\underline{\Gamma}(b^*) = 0$.  
Here, there are two scenarios:
\begin{enumerate}
	\item $\underline{\Gamma}(b^*) =\min_{0 \leq a \leq b^*} \Gamma(a,b^*)= 0$ is attained at a local minimizer $a^* = b^* - \epsilon$;
	\item $\underline{\Gamma}(b^*) =\min_{0 \leq a \leq b^*} \Gamma(a,b^*) = 0$ is attained at zero $a^* = 0$.
\end{enumerate}
In both cases, we have $\Gamma(a^*, b^*) = 0$ (or $\mathbf{C_b}$).  In addition, for (1), $\gamma(a^*,b^*)
= 0$ holds as well, while for (2), we must have that $\Gamma(\cdot, b^*)$ is monotonically increasing (or $\partial \Gamma(\cdot, b^*) / {\partial a}  \geq 0$) and hence \eqref{a_b_zero_equivalence} gives the inequality: 
\begin{align*}
\gamma(0+, b^*)
= \frac r {q+r}   \Big( 1 - \frac 1 q    \Big( \beta (q+r) - r \Big) Z^{(q+r)}(b^*)  \Big)\geq 0.
\end{align*}


We summarize these in the following lemma.
\begin{lemma} \label{lemma_existence}
	There exist a  pair $0 \leq a^*  < b^*$ such that one of the following holds. 
	\begin{enumerate}
		\item $0 < a^* < b^*$ such that $\Gamma(a^*, b^*) = \gamma(a^*,b^*)
		 = 0$.
		\item $a^* = 0 < b^*$ such that $\Gamma(a^*, b^*) = 0$ and $Z^{(q)}(b^*-a^*)= Z^{(q)}(b^*) \leq q/(\beta(q+r)-r)$.		
	\end{enumerate}
\end{lemma}

\subsection{Convergence results} \label{subsection_convergence}

We conclude this section with convergence results with respect to $\beta$ and $r$. The following results indicate that the optimal hybrid barrier strategies approach the pure continuous/periodic barrier strategies, which are shown  in the following section to be optimal when Assumption \ref{assumption_beta_r_q} does not hold.

\begin{lemma} \label{lemma_convergence_beta_r}
\begin{enumerate}
\item As $\beta \downarrow r /(q+r)$, we have $b^*  \rightarrow \infty$.
\item As $r \uparrow q \beta / (1 - \beta)$ (and then $r / (q+r) \uparrow \beta$), we have $b^*  \rightarrow \infty$.
\item As $\beta \uparrow  1$, we have  $b^* - a^* \rightarrow 0$.
\end{enumerate}
\end{lemma}
\begin{proof}
(1) and (2): For $b > 0$, we have
\begin{align*}
\Gamma(0,b) 
=  \frac {\beta \psi'(0+)} {q}   + \frac {q+r} q \Big(\beta - \frac r {q+r} \Big)  \overline{Z}^{(q+r)}(b). \end{align*}
This, Assumption  \ref{assumption_beta_r_q} and \eqref{Gamma_zero_inf} imply that $b \mapsto \Gamma(0,b)$ starts at a negative value $\Gamma(0,0)$ (see \eqref{Gamma_zero_inf}) and strictly  increases to infinity.
Hence, we can define its unique root $b_0 > 0$ such that
$\Gamma(0, b_0) = 0$.

When $a(b_0) > 0$, it is clear that $b_0 < b^*$ because, by Lemma \ref{lemma_Gama_wrt_a},
$\underline{\Gamma} (b_0) = \Gamma (a(b_0), b_0)  < \Gamma(0, b_0) = 0 = \underline{\Gamma} (b^*)$
and $\underline{\Gamma}$ is a strictly increasing function by Lemma \ref{Gamma_inf_increasing}.

On the other hand, when $a(b_0) = 0$, then it is clear that $b_0 = b^*$ (and $a^* = 0$). In sum, we have 
\begin{align}
b_0 \leq b^*. \label{b_0_b_star_relation}
\end{align}

Now, because
\begin{align*}
0 < - \frac {\beta \psi'(0+)} {q}   =  \frac {q+r} q \Big(\beta - \frac r {q+r} \Big)  \overline{Z}^{(q+r)}(b_0). \end{align*}
we have $b_0 \rightarrow \infty$ as $\beta \downarrow r /(q+r)$ and as $r \uparrow q \beta / (1 - \beta)$. By \eqref{b_0_b_star_relation}, we have (1) and (2).

(3) From \eqref{def-epsilon} and Lemma \ref{lemma_existence}, it is clear that $b^* - a^* \leq \epsilon$.
Hence, as $\beta \uparrow  1$, it is clear from \eqref{def-epsilon} that $\epsilon \rightarrow 0$ and hence $b^* - a^* \rightarrow 0$.
\end{proof}

\section{Optimality of the hybrid barrier strategies} \label{section_optimality}
In this section, we show the optimality of $\pi_{a^*, b^*}$ under Assumption \ref{assumption_beta_r_q}.
%
By Lemma \ref{lemma_existence},  $\mathbf{C}_b$ is always satisfied and hence, by \eqref{vf_ff},
\begin{multline} \label{value_function_final}
v_{a^*,b^*}(x)
=
-\Gamma(a^*,b^*;x) = - \frac r {q+r} \overline{Z}^{(q)}(a^*-x)- \frac {\beta \psi'(0+)} {q}  \\  - \Big(\beta - \frac r {q+r} \Big) \Big[ \overline{Z}^{(q,r)}_{a^*-b^*}(a^*-x) + \frac r q  Z^{(q)}(a^*-x)    \overline{Z}^{(q+r)}(b^*-a^*) \Big]. 
\end{multline}
We show that this solves the variational inequalities given in Lemma \ref{verificationlemma_2}.

Recall the mapping $x \mapsto \mathcal{L} g(x)$ as in \eqref{generator} for any sufficiently smooth function $g$ of $x$.
\begin{lemma}\label{ver_1}
We have,
for $x \geq b^*$,
\begin{align*}
(\mathcal{L}-q)v_{a^*,b^*}(x) =\frac {qr} {q+r} (a^*-x) - \Big(\beta - \frac r {q+r}  \Big) \Big[ q (x-b^*) -r  \overline{Z}^{(q+r)}(b^*-a^*)\Big], 
\end{align*}
for $a^* < x < b^*$,
\begin{align*}
(\mathcal{L}-q)v_{a^*,b^*}(x) =r \Big[\frac {q} {q+r} (a^*-x) - \Big(\beta - \frac r {q+r}  \Big) \Big(  \overline{Z}^{(q+r)}(b^*-x)-  \overline{Z}^{(q+r)}(b^*-a^*)\Big) \Big], 
\end{align*}
and, for $0 < x \leq a^*$,  $(\mathcal{L}-q)v_{a^*,b^*}(x) = 0$.

\end{lemma}


\begin{proof}For $x > 0$, by \eqref{value_function_final},
\begin{multline} \label{generator_v_a_b_premitive}
(\mathcal{L}-q) v_{a^*,b^*}(x)
= - \frac r {q+r} (\mathcal{L}-q) ( \overline{Z}^{(q)}(a^*-x)) +  {\beta \psi'(0+)}   \\ - \Big( \beta - \frac r {q+r}  \Big) \Big[ (\mathcal{L}-q)(\overline{Z}^{(q,r)}_{a^*-b^*}(a^*-x)) + \frac r q  \overline{Z}^{(q+r)}(b^*-a^*)  (\mathcal{L}-q)(Z^{(q)}(a^*-x))    \Big].
\end{multline}
Here, by e.g. the results in the proof of Theorem 2.1 in \cite{BKY}, for $x \neq a^*$,
\begin{align*}
(\mathcal{L}-q) (\overline{Z}^{(q)}(a^*-x)) = \psi'(0+) + q (x - a^*) 1_{\{ x > a^*\}}  \quad \textrm{and} \quad
(\mathcal{L}-q) (Z^{(q)}(a^*-x)) =  -q 1_{\{ x > a^*\}}.\end{align*}
In addition, for $x > a^*$, 
\begin{align*}
 (\mathcal{L}-q) (\overline{Z}^{(q,r)}_{a^*-b^*}(a^*-x)) = (\mathcal{L}-q) (\overline{Z}^{(q+r)}(b^*-x)) = \left\{ \begin{array}{ll}  q (x-b^*)  + \psi'(0+) & x \geq b^*, \\  r\overline{Z}^{(q+r)}(b^*-x) + \psi'(0+) & x < b^*,\end{array} \right.
\end{align*}
and, for $x < a^*$, by using the result in the proof of Lemma 4.5 of \cite{Yamazaki2013},
	\begin{align*}
		(\mathcal{L}-q)(\overline{Z}_{a^*-b^*}^{(q,r)}(a^*-x))&=(\mathcal{L}-q)\left(\overline{Z}^{(q+r)}(b^*-x)-r\int_0^{a^*-x}W^{(q)}(a^*-x-y)\overline{Z}^{(q+r)}(y+b^*-a^*) \diff y\right)\\
		&= r\overline{Z}^{(q+r)}(b^*-x) + \psi'(0+)-r\overline{Z}^{(q+r)}(b^*-x)
		=\psi'(0+).
	\end{align*}
	Substituting these in \eqref{generator_v_a_b_premitive}, we have the claim.
%
\end{proof}

By how $a^*$ and $b^*$ are chosen as in Lemma \ref{lemma_existence}, we have the following.
\begin{lemma} \label{lemma_slopes}
\begin{enumerate}
\item We have $v_{a^*,b^*}'(a^*) = 1$ when $a^* > 0$ and $v_{a^*,b^*}'(0+) \leq 1$ when $a^* = 0$.
\item  We have $v_{a^*,b^*}'(b^*) = \beta$.
\item The function $v_{a^*,b^*}$ is concave on $(0, \infty)$.
\end{enumerate}
\end{lemma}
\begin{proof}
(1) 
By differentiating \eqref{value_function_final},
$v_{a^*,b^*}'(x)
=
-\frac \partial {\partial x}\Gamma(a^*,b^*;x)$,
where, in particular,  equation \eqref{Gamma_derivatives} gives
\begin{align*}
v_{a^*,b^*}'(a^*)
 &= \frac r {q+r}  + \Big(\beta - \frac r {q+r}  \Big)  Z^{(q+r)}(b^*-a^*). 
 \end{align*}
For the case $a^* > 0$, because $\mathbf{C}_a$, given by \eqref{cond_2}, is satisfied by Lemma \ref{lemma_existence} (1),
\begin{align*}
v_{a^*,b^*}'(a^*) = \frac r {q+r} + \Big(\beta - \frac r {q+r}  \Big) \frac q {\beta(q+r)-r} = 1.
\end{align*}
For the case $a^* = 0$, by Lemma \ref{lemma_existence} (2),
\begin{align*}
v_{a^*,b^*}'(0+) 
&\leq \frac r {q+r} + \Big(\beta - \frac r {q+r}  \Big) \frac q {\beta(q+r)-r} = 1.
\end{align*}
(2) By \eqref{Gamma_derivatives}, it is clear that $v_{a^*,b^*}'(x) = \beta$ for all $x\geq b^*$.\\
(3) (i) For $x < a^*$, integration by parts  applied to the third equality of Remark \ref{remark_Z_q_r_another_expression} gives
\begin{align*}
		\overline{Z}_{a^*-b^*}^{(q,r)}(a^*-x) 
		&=\overline{Z}^{(q)}(b^*-x) - r \overline{W}^{(q)} (a^*-x) \overline{Z}^{(q+r)}(b^*-a^*) +r\int_0^{b^*-a^*}\overline{W}^{(q)}(b^*-x-u)Z^{(q+r)}(u) \diff u.
\end{align*}
By differentiating this,
\begin{align*}
		\frac \partial {\partial x}\overline{Z}_{a^*-b^*}^{(q,r)}(a^*-x) 
		&=-Z^{(q)}(b^*-x)+ r W^{(q)} (a^*-x) \overline{Z}^{(q+r)}(b^*-a^*) \\&-r\int_0^{b^*-a^*} W^{(q)}(b^*-x-u)Z^{(q+r)}(u) \diff u, \\
				\frac {\partial^2} {\partial x^2}\overline{Z}_{a^*-b^*}^{(q,r)}((a^*-x)+) 
		&=q W^{(q)}(b^*-x)- r W^{(q) \prime} ((a^*-x)+) \overline{Z}^{(q+r)}(b^*-a^*) \\&+r\int_0^{b^*-a^*} W^{(q)\prime}(b^*-x-u)Z^{(q+r)}(u) \diff u.
\end{align*}
Hence, by \eqref{value_function_final} and Assumption \ref{assumption_beta_r_q},
\begin{align*}
		v_{a^*,b^*}''(x)&=-\frac{rq}{r+q}W^{(q)}(a^*-x)\\&-\Big( \beta - \frac r {q+r} \Big)\left(qW^{(q)}(b^*-x)+r\int_0^{b^*-a^*}W^{(q)\prime}(b^*-x-u)Z^{(q+r)}(u) \diff u\right) \leq 0.
	\end{align*}
	(ii)  For $x > a^*$, by \eqref{Gamma_derivatives} and Assumption \ref{assumption_beta_r_q},
	\begin{align*} 
v_{a^*,b^*}''(x) = -\frac {\partial ^2} {\partial x^2}\Gamma(a^*,b^*;x) 
 &= - \Big(\beta - \frac r {q+r}  \Big)  (q+r)W^{(q+r)}(b^*-x) \leq 0.
\end{align*}
By (i) and (ii), together with the continuity of $v_{a^*, b^*}'$  as in Lemma \ref{lemma_smooth_fit_cond}, $v_{a^*, b^*}$ is concave. 
\end{proof}

\begin{lemma}\label{ver_3} For $x \geq 0$, we have
\begin{align*}
m(x) &:= \max_{0\leq l\leq x} \{ l+v_{a^*,b^*}(x-l)-v_{a^*,b^*}(x) \} \\&=
	\left[\frac q {q+r}(x-a^*) - \Big(\beta - \frac r {q+r} \Big) \left( \overline{Z}^{(q+r)}(b^*-a^*) - \overline{Z}^{(q+r)}(b^*-x)\right)\right]1_{\{x>a^*\}}.
\end{align*}
\end{lemma}
\begin{proof}
By the slope condition and the concavity of $v_{a^*, b^*}$ as in Lemma \ref{lemma_slopes} (see also Lemma 3.1 in \cite{ATW}),
	\begin{equation}\label{max_cond}
	m(x)=
	\begin{cases} 0 &\mbox{if } x \in[0,a^*], \\ 
	x-a^*+v_{a,^*b^*}(a^*)-v_{a^*,b^*}(x) & \mbox{if } x \in (a^*,\infty). \end{cases}
	\end{equation}
Hence it is left to show for the case  $x>a^*$.  Indeed, by \eqref{value_function_final}, 
\begin{align*}
	v_{a^*,b^*}(a^*) - v_{a^*,b^*}(x)&=  -\frac r {q+r} (x-a^*)  - \Big(\beta - \frac r {q+r}  \Big)  \left(\overline{Z}^{(q+r)} (b^*-a^*)- \overline{Z}^{(q+r)}(b^*-x)\right),
\end{align*}
and hence
\begin{align*}
	m(x) = \frac q {q+r}(x-a^*) - \Big(\beta - \frac r {q+r}  \Big) \left( \overline{Z}^{(q+r)}(b^*-a^*) - \overline{Z}^{(q+r)}(b^*-x)\right).
\end{align*}
\end{proof}

By Lemmas  \ref{ver_1}, \ref{lemma_slopes}, and \ref{ver_3}, the main result is immediate.
\begin{theorem}
The function $v_{a^*, b^*}$ solves \eqref{HJB-inequality_db} and \eqref{HJB-inequality_db2}.  Hence $\pi_{a^*, b^*}$ is the optimal strategy.
\end{theorem}
\begin{proof}

	
By Lemmas \ref{ver_1} and \ref{ver_3} and Assumption \ref{assumption_beta_r_q},
	\begin{align*}
	 (\mathcal{L} - q)v_{a^*, b^*}(x)+r m(x)=  \left\{ \begin{array}{ll}\big(\beta - \frac r {q+r}  \big) (q+r)  (b^*-x) &  \textrm{if } x \geq b^* \\
	0 & \textrm{if }  0 < x < b^* 
	\end{array}\right\} \leq 0,
	\end{align*}
	and hence $v_{a^*, b^*}$ solves \eqref{HJB-inequality_db} .  On the other hand, it solves  \eqref{HJB-inequality_db2} by Lemma \ref{lemma_slopes}.  Hence, by Lemma \ref{verificationlemma_2}, the optimality holds.

\end{proof}
\section{Optimality of pure periodic and continuous strategies} \label{section_other_cases}
In this section, we consider the case Assumption \ref{assumption_beta_r_q} does not hold.  We show when $0\leq \beta\leq {r} / {(r+q)}$ that a pure periodic barrier strategy as defined in Section \ref{subsection_various_strategies} is optimal; when $\beta \geq 1$, a pure continuous barrier strategy is optimal.
\subsection{Optimality of the pure periodic barrier strategy
} 
Here we assume
$0\leq \beta\leq r / (r+q)$, 
and show that it is optimal not to activate the continuous strategy ($b^*=\infty$). 

Let $\mathcal{A}_p \subset \mathcal{A}$ be the set of strategies such that $L^\pi_c \equiv 0$ uniformly.
P\'erez and Yamazaki \cite{YP2016} studied the optimization problem over the restricted set $\mathcal{A}_p$.  By Section 4.3 in \cite{YP2016}, we have 
\begin{align*}
\sup_{\pi \in \mathcal{A}_p}v_{\pi}(x) = v_{a^*_p,\infty}(x), \quad x > 0, 
\end{align*}
where $v_{a^*_p,\infty}$ is the value function under the periodic barrier strategy with the barrier $a^*_p \geq 0$, which is given as follows: 
\begin{enumerate}
\item For the case $\psi'(0+)<-{q(r+q)} /{(r\Phi(q+r))}$, we have
\begin{align} \label{value_periodic_a_bigger_0}
\begin{split}
	v_{a^*_p,\infty}(x) &=-\frac{r}{r+q}\left(\overline{Z}^{(q)}(a^*_p-x)+\frac{\psi'(0 +)}{q}\right)\\&-\frac{1}{\Phi(q+r)}\left(\frac{r}{r+q}Z^{(q)}(a^*_p-x)+\frac{q}{r+q} Z^{(q)}(a^*_p-x, \Phi(q+r))  \right), \quad x \geq 0,
	\end{split}
\end{align}
where $a^*_p$ is the unique root of
\begin{equation*}
-\Phi(q+r) \frac{r}{r+q}\left(\overline{Z}^{(q)}(a)+\frac{\psi'(0+)}{q}\right)= \frac{r}{r+q} Z^{(q)}(a)
		+\frac{q}{r+q} Z^{(q)}(a, \Phi(q+r)).
\end{equation*}
\item If $\psi'(0+) \geq -{q(r+q)} /{(r\Phi(q+r))}$, then $a^*_p = 0$ and
\begin{align} \label{value_periodic_a_0}
	v_{0,\infty}(x)=\frac{r}{r+q}\left[x-\frac{\psi'(0+)}{r+q}\left(1-{\rm e}^{-\Phi(q+r)x}\right)\right], \quad x \geq 0.
\end{align}
\end{enumerate}

Now, in order to show that it is indeed optimal over the unrestricted set $\mathcal{A}$, it suffices to show that $v_{a^*_p, \infty}$ solves the  variational inequalities 
\eqref{HJB-inequality_db} and \eqref{HJB-inequality_db2}. By Theorem 4.1 and Lemma 4.3 in \cite{YP2016}, the inequality \eqref{HJB-inequality_db} is already known to hold, and hence it remains to show \eqref{HJB-inequality_db2}.

By differentiating \eqref{value_periodic_a_bigger_0} and \eqref{value_periodic_a_0} and taking limits,
$\lim_{x\to\infty}v_{a^*_p,\infty}'(x)= {r} /{(r+q)} \geq\beta$.
In addition, the concavity of $v_{a^*_p, \infty}$ holds when $a^*_p > 0$ by Lemma 4.4 (1) of \cite{YP2016}; it also holds when $a^*_p = 0$ by \eqref{value_periodic_a_0} and our assumption \eqref{assumption_drift}. Hence, for $x > 0$, we conclude that
$v'_{a^*_p,\infty}(x)\geq r/(r+q)\geq \beta$,
as desired.
\subsection{Optimality of the pure continuous barrier strategy}
We now assume that
$\beta\geq1$,
and show that it is optimal not to exercise 
the periodic strategy (i.e. $a^*=b^*$).

Let $\mathcal{A}_c \subset \mathcal{A}$ be the set of strategies such that $L^\pi_p \equiv 0$ uniformly.
Bayraktar et al.\ \cite{BKY} studied the optimization problem over the restricted set $\mathcal{A}_c$.  By  equation (2.5) in \cite{BKY} (with normalization), we have 
\begin{align}
\sup_{\pi \in \mathcal{A}_c}v_{\pi}(x) = v_{b^*_c,b^*_c}(x) = -\beta\left(\overline{Z}^{(q)}(b^*_c-x)+\frac{\psi'(0+)}{q}\right), \quad x > 0,  \label{value_function_continuous}
\end{align}
where
\begin{align}
b^*_c:=(\overline{Z}^{(q)})^{-1}(-\psi'(0+)/q). \label{b_star_classical}
\end{align}


Again, we shall show that $v_{b^*_c, b^*_c}$ solves the variational inequalities \eqref{HJB-inequality_db} and \eqref{HJB-inequality_db2}.

Because for $0 \leq l \leq x$,
\begin{align*}
\frac \partial {\partial l} \big[ l+v_{b^*_c,b^*_c}(x-l) \big] = \frac \partial {\partial l} \Big[ l-\beta\left(\overline{Z}^{(q)}(b^*_c-x+l)+\frac{\psi'(0+)}{q}\right) \Big] 
=1-\beta Z^{(q)}(b^*_c-x+l)\leq 1-\beta\leq0,
\end{align*}
we have
$
\max_{0\leq l\leq x} \{ l+v_{b^*_c,b^*_c}(x-l)-v_{b^*_c,b^*_c}(x) \}=0$ for $x > 0.
$
This together with $(\mathcal{L}-q)v_{b^*_c,b^*_c}(x)\leq 0$ (by Theorem 2.1 in \cite{BKY}) shows \eqref{HJB-inequality_db}.

On the other hand, \eqref{HJB-inequality_db2} holds because $v_{b^*_c,b^*_c}'(x)=\beta Z^{(q)}(b^*_c-x)\geq \beta$ for all $x > 0$.

\section{Numerical Results} \label{section_numerics}

%

In this section, we confirm the analytical results obtained in the previous sections through a sequence of numerical experiments.  Here, we assume that the underlying process $X$ is a spectrally positive version of the \emph{phase-type} \lev process (with Brownian motion) as discussed in \cite{Asmussen_2004}. Its scale function can be expressed analytically as discussed in \cite{Egami_Yamazaki_2010_2}, and this process is particularly important because it can approximate any  spectrally positive \lev process (see \cite{Asmussen_2004} and \cite{Egami_Yamazaki_2010_2}).  

We assume that 
\begin{equation}
 X(t) - X(0)= - c t+\sigma B(t) + \sum_{n=1}^{N(t)} Z_n, \quad 0\le t <\infty, \label{X_phase_type}
\end{equation}
where, $c \in \R$, $\sigma > 0$,  $B=( B(t); t\ge 0)$ is standard Brownian motion, and $N=(N(t); t\ge 0 )$ is a Poisson process with arrival rate $\kappa$. In addition,  $Z = ( Z_n; n = 1,2,\ldots )$ is an i.i.d.\ sequence of phase-type-distributed random variables with representation $(m,{\bm \alpha},{\bm T})$, or equivalently the first absorption time in a continuous-time Markov chain consisting of a single absorbing state and $m$ transient states with initial distribution $\alpha$ and transition matrix ${\bm T}$ (see \cite{Asmussen_2004} for details). The processes $B$, $N$, and $Z$ are assumed to be mutually independent.  We refer the reader to \cite{Egami_Yamazaki_2010_2, KKR} for the forms of the corresponding scale functions.

\subsection{Computing $(a^*, b^*)$ and the value function $v_{a^*,b^*}$}  For the process $X$ in \eqref{X_phase_type}, we set $\sigma = 0.2$ and $\kappa = 2$, and, for $Z$, we use a phase-type distribution with  $m=6$ 
that approximates a (folded) normal random variable with mean $0$ and variance $1$ (see \cite{leung2015analytic} for the values of $\alpha$ and ${\bm T}$).  For the drift parameter $c$, we consider \textbf{Case 1:} $c=1$  and \textbf{Case 2:} $c=1.5$,
 obtaining the cases $a^* > 0$ and $a^* = 0$, respectively.  For the other parameters, let $q = 0.05$, $r = 0.05$, and $\beta = 0.6$ so that Assumption \ref{assumption_beta_r_q} is satisfied.

The first implementation step is computing the optimal barriers $(a^*, b^*)$ as described in Lemma \ref{lemma_existence}.  
As discussed in Section \ref{subsection_existence}, we choose these such that $\underline{\Gamma}(b^*) = 0$ and $a^* = \arg \min_{0 \leq a \leq b^*} \Gamma(a,b^*)$.   Thanks to Lemma \ref{lemma_Gama_wrt_a} (i.e., the fact that $a \mapsto \gamma(a,b) := \partial \Gamma(a,b) / \partial a$ is monotone), the minimizer $a(b)$ of $a \mapsto \Gamma(a,b)$ can be obtained by the bisection method.  In addition, using Lemma 
\ref{Gamma_inf_increasing}, another bisection can be used to obtain $b^*$ (and $a^* = a(b^*)$).  With these $a^*$ and $b^*$ values, the value function becomes $v_{a^*, b^*}$, as in \eqref{value_function_final}.

In the first two rows of Figure \ref{plot_gamma}, we plot, for \textbf{Cases 1} and \textbf{2}, the mappings $a \mapsto \Gamma(a,b)$ and $a \mapsto \gamma(a,b)$ for various values of $b$.  As in \eqref{Gamma_derivative_b_uniform_bound}, $b \mapsto \Gamma(a,b)$ is uniformly increasing. 
The values $a(b)$ correspond to the points (indicated by the down-pointing triangles) at which $\Gamma(\cdot, b)$ is minimized and $\gamma(\cdot, b)$ vanishes (if $a(b) > 0$). 
The optimal pair is such that the minimum $\underline{\Gamma}$ is zero.
In Figure \ref{plot_value_function},  we plot the corresponding value functions $v_{a^*, b^*}$ (solid lines) along with the suboptimal NPVs $v_{a,b}$ (dotted lines) given in  \eqref{vf_ff} where $(a,b) \neq (a^*, b^*)$.  It can be confirmed in both cases that $v_{a^*, b^*}$ dominates $v_{a,b}$,  for $(a,b) \neq (a^*, b^*)$, uniformly in $x$.
%
%
 \begin{figure}[htbp]
\begin{center}
\begin{minipage}{1.0\textwidth}
\centering
\begin{tabular}{cc}
 \includegraphics[scale=0.4]{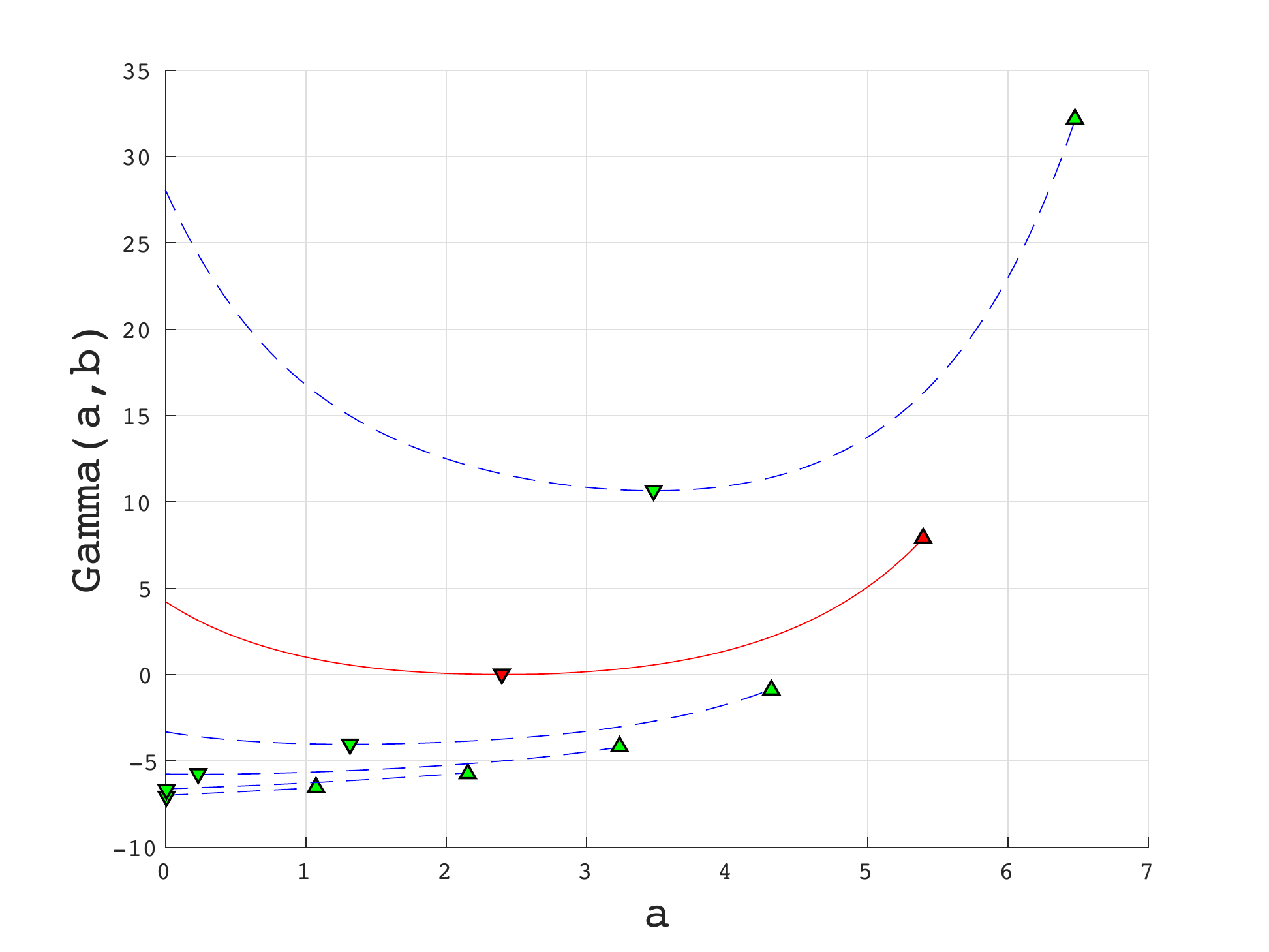} & \includegraphics[scale=0.4]{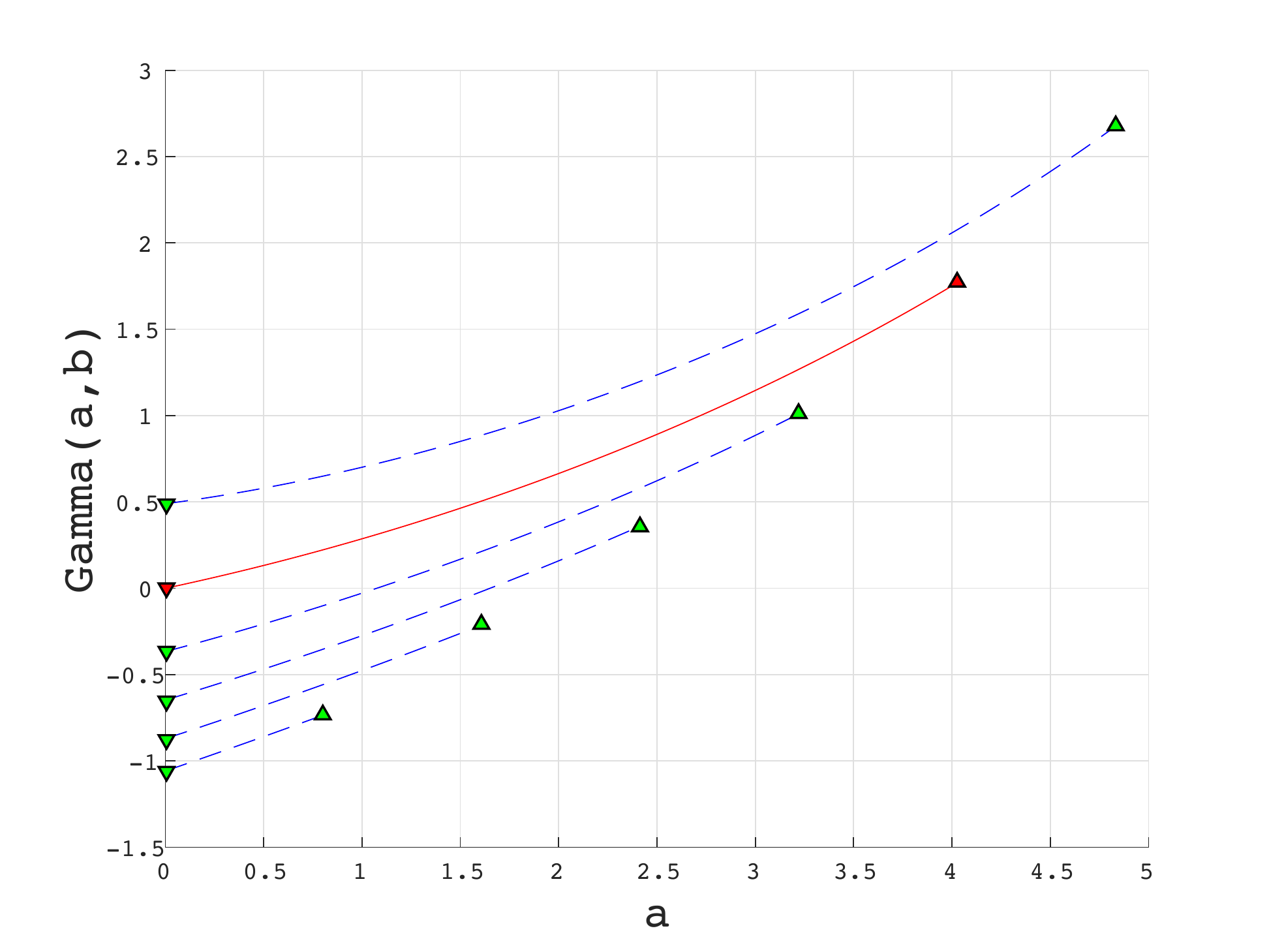}  \\
   \includegraphics[scale=0.4]{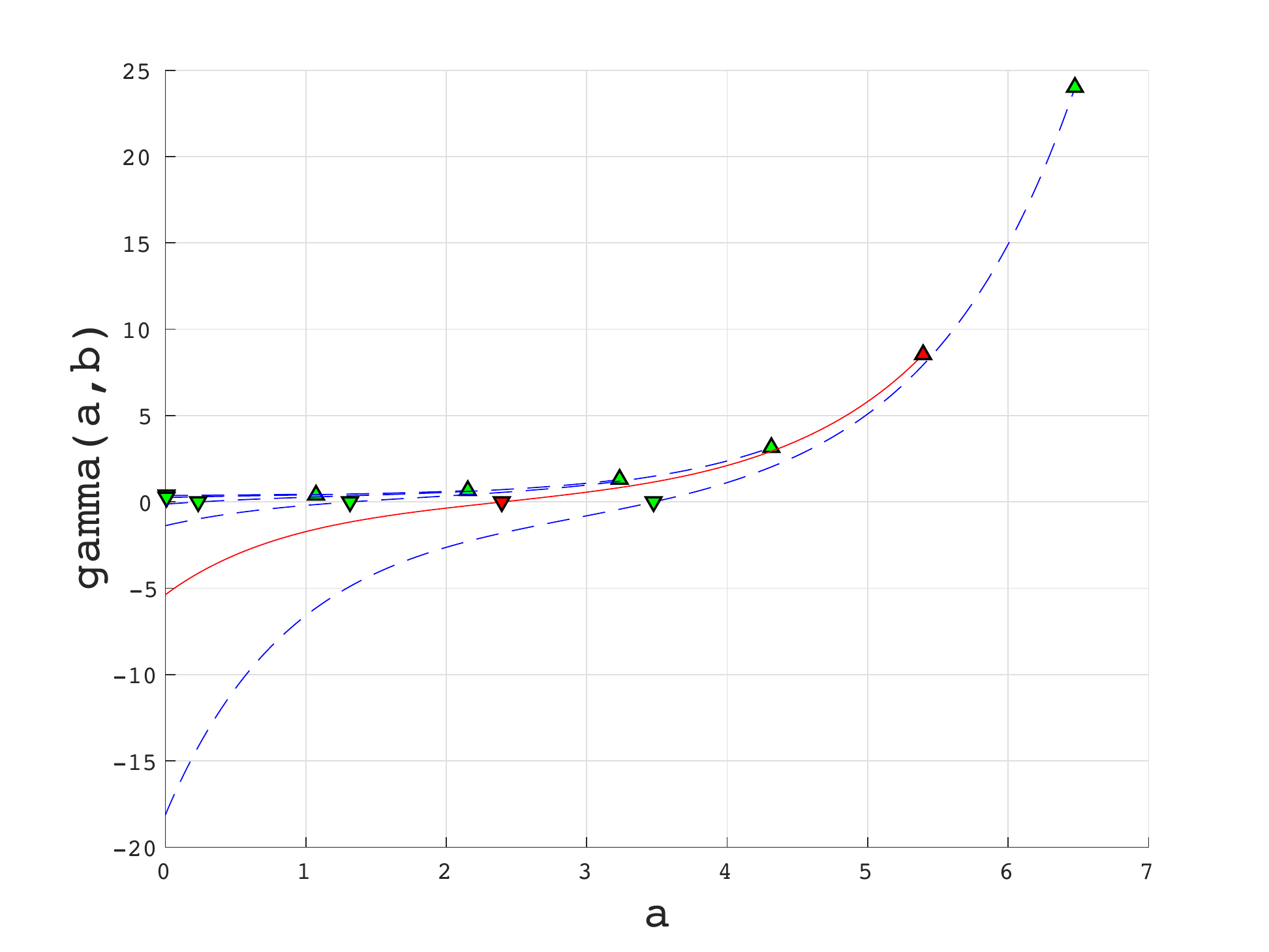} & \includegraphics[scale=0.4]{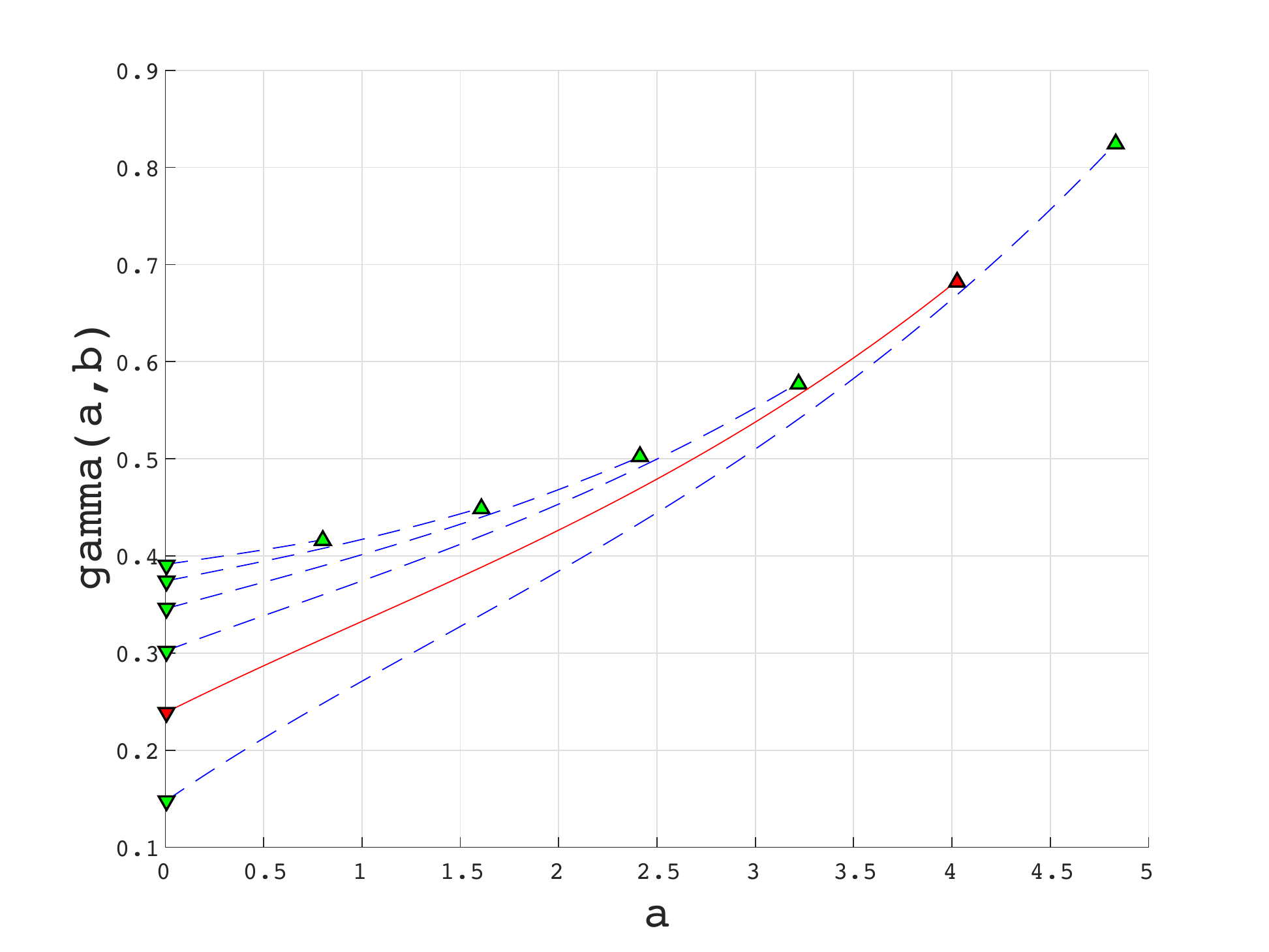}    \\  
 Case 1 & Case 2 \end{tabular}
\end{minipage}
\caption{Plots of $a \mapsto \Gamma(a, b)$ (top) and $a \mapsto \gamma(a, b)$ (bottom)  for \textbf{Case 1} (left) and \textbf{Case 2} (right), for $b = b^*/5, 2b^*/5, \ldots, b^*, 6b^*/5$. The solid red lines correspond to the cases where $b = b^*$.    The down- and up-pointing triangles indicate the points at $a = a(b)$ and $a=b$, respectively.  } \label{plot_gamma}
\end{center}
\end{figure}

\begin{figure}[htbp]
\begin{center}
\begin{minipage}{1.0\textwidth}
\centering
\begin{tabular}{cc}
 \includegraphics[scale=0.4]{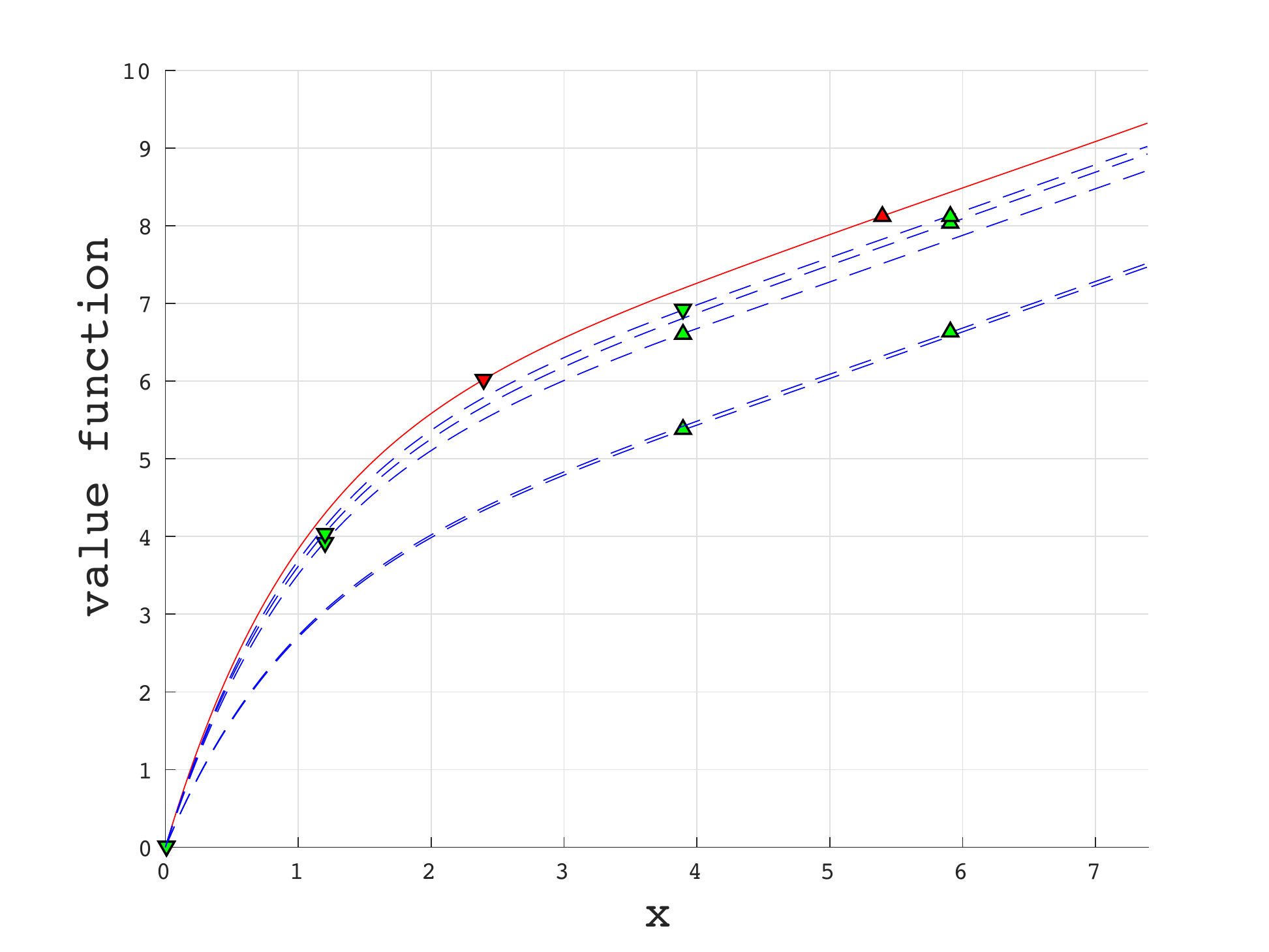} & \includegraphics[scale=0.4]{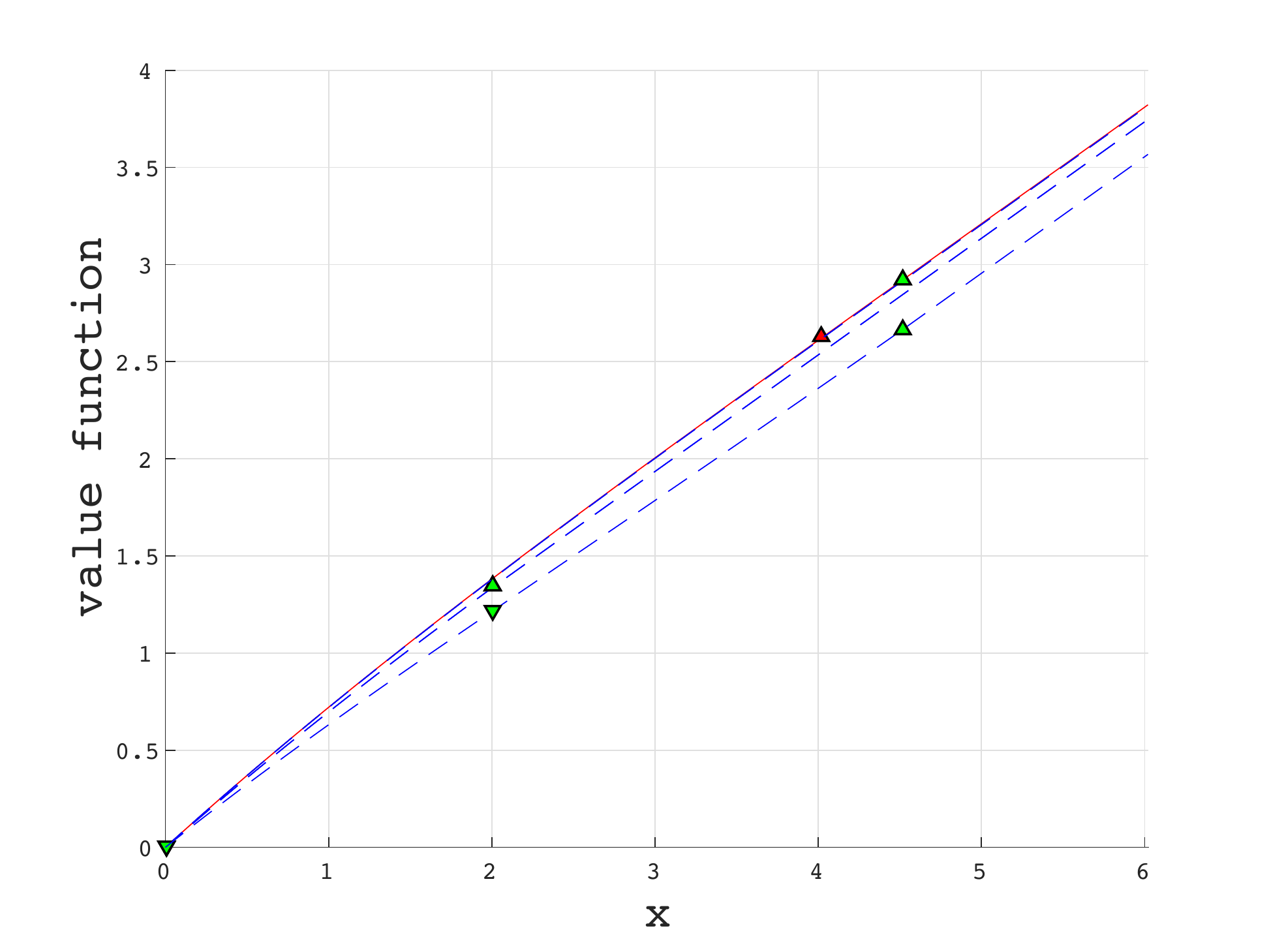}  \\
Case 1 & Case 2 \end{tabular}
\end{minipage}
\caption{Plots of $v_{a^*, b^*}$  (solid lines) along with the suboptimal expected NPVs $v_{a,b}$ (dotted lines) for $a = 0,a^*/2, (a^* + b^*)/2$ and $b = (a^*+b^*)/2,b^*+0.5$.  The up- and down-pointing triangles indicate the values at $a$ and $b$, respectively.} \label{plot_value_function}
\end{center}
\end{figure}

\subsection{Sensitivity and convergence with respect to $\beta$ and $r$}
Next, we study the behavior of the value function $v_{a^*, b^*}$ with respect to $\beta$ and $r$ to confirm the results discussed in Sections \ref{subsection_convergence} and \ref{section_other_cases}.  Here we use the same parameters as in \textbf{Case 1} above unless stated otherwise.

On the left of Figure \ref{plot_convergence}, we plot the value function $v = v_{a^*, b^*}$, for $\beta \in (r / (q+r), 1)$, along with $v=v_{a^*_p, \infty}$ and $v=v_{b^*_c, b^*_c}$, as in \eqref{value_periodic_a_bigger_0} and \eqref{value_function_continuous}, respectively, for the cases $\beta \leq r /(q+r)$ and $\beta = 1$.   This confirms that $v(x)$ increases with $\beta$ for each $x > 0$. As in Lemma \ref{lemma_convergence_beta_r} (1), as $\beta \downarrow r / (q+r)$, the value of $b^*$ goes to infinity and consequently  $v_{a^*, b^*}$ converges decreasingly to $v_{a^*_p, \infty}$ (which is independent of the value of $\beta$).  In contrast, as in Lemma \ref{lemma_convergence_beta_r} (3), as $\beta \uparrow 1$, the  distance between $a^*$ and $b^*$ shrinks to zero and consequently the value function $v_{a^*, b^*}$ converges increasingly to $v_{b^*_c, b^*_c}$ (when $\beta = 1$).

On the right of Figure \ref{plot_convergence}, we plot the value function $v = v_{a^*, b^*}$, for $r \in (0, q \beta / (1 - \beta))$, along with the value functions $v=v_{a^*_p, \infty}$ and $v=v_{b^*_c, b^*_c}$, as in \eqref{value_periodic_a_bigger_0} and \eqref{value_function_continuous}, respectively, for the cases $r = q \beta / (1 - \beta)$ (equivalently $\beta = r /(q+r)$) and $r = 0$.  This confirms that $v(x)$ increases with $r$ for each $x > 0$. As in Lemma \ref{lemma_convergence_beta_r} (2), as $r \uparrow  q \beta / (1 - \beta)$, the value of $b^*$ increases to infinity and consequently the value function $v_{a^*, b^*}$ converges increasingly to $v_{a^*_p, \infty}$ (which is independent of the value of $\beta$).  In contrast, as $r \downarrow 0$, periodic dividend opportunities vanish and hence the problem gets closer to the one defined in  \eqref{value_function_continuous}.  Consequently $b^*$ and $v_{a^*, b^*}$ converge to the values defined in \eqref{b_star_classical} and \eqref{value_function_continuous}, respectively. 


\begin{figure}[htbp]
\begin{center}
\begin{minipage}{1.0\textwidth}
\centering
\begin{tabular}{cc}
 \includegraphics[scale=0.4]{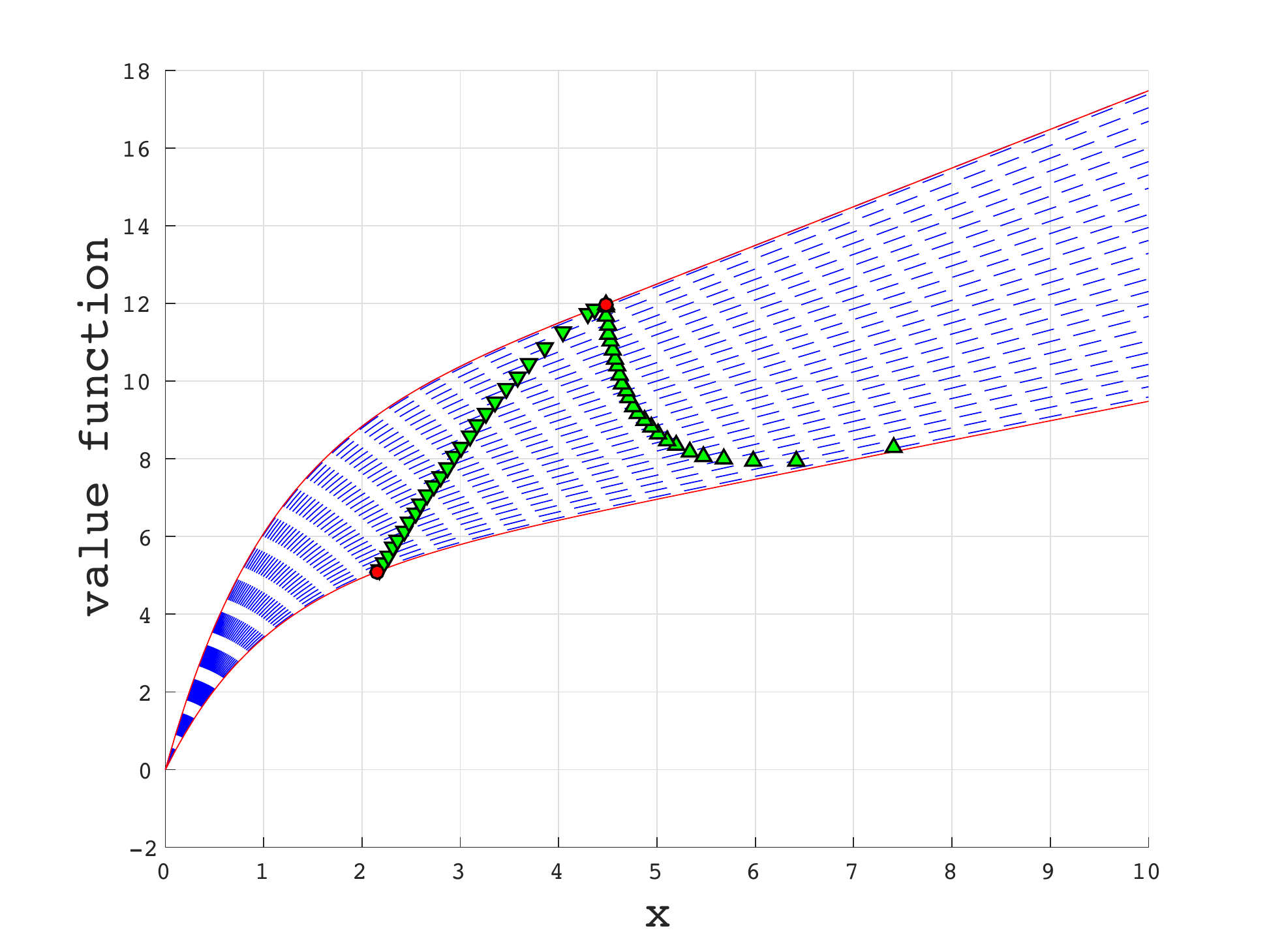} & \includegraphics[scale=0.4]{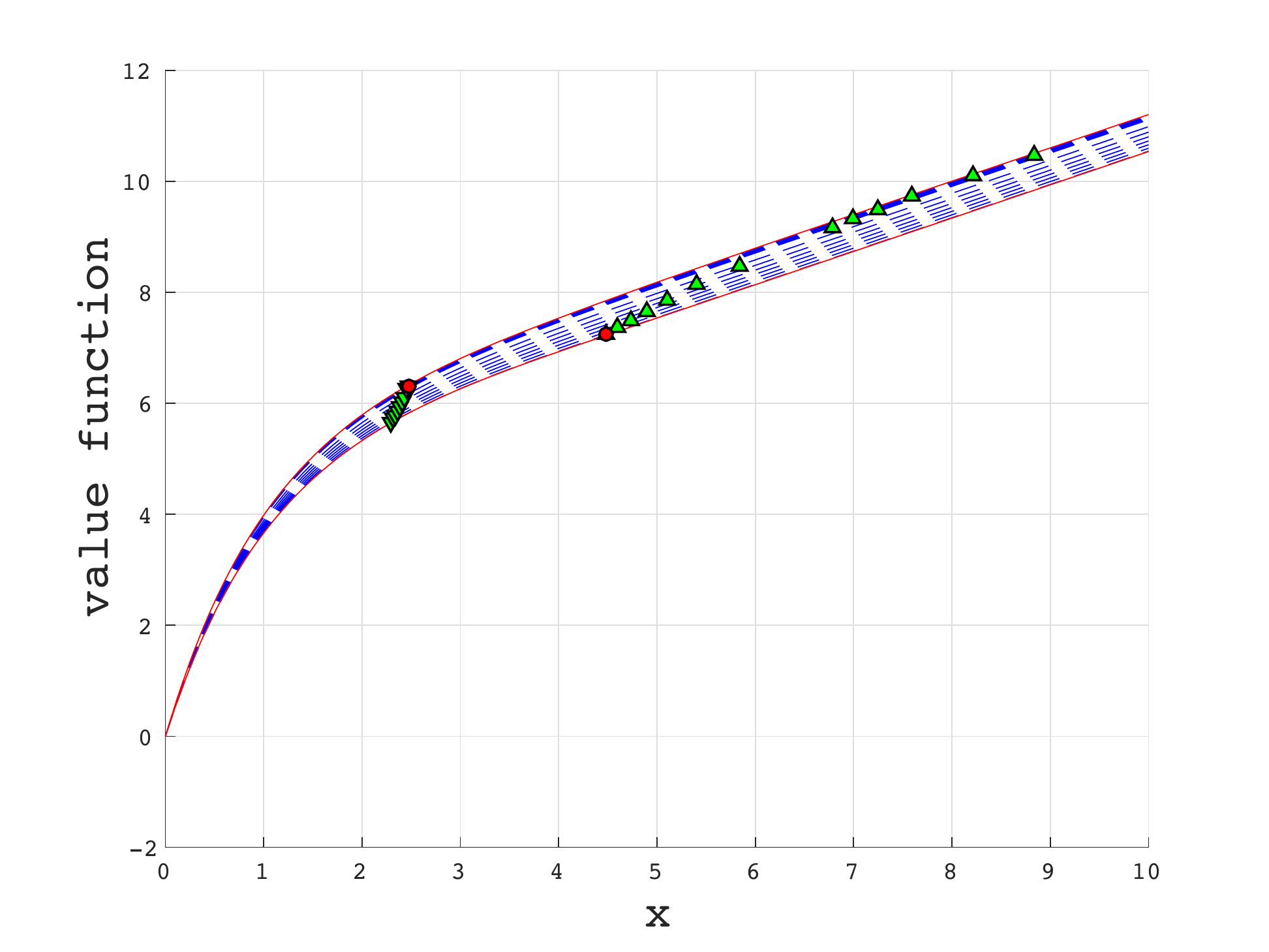}  \\
 convergence in $\beta$ & convergence in $r$ \end{tabular}
\end{minipage}
\caption{(Left) The plots of the value functions $v_{a^*,b^*}$ (dotted lines) for $\beta = 0.51$, $0.53$, ..., $0.97,$ $0.99$, $0.995$, along with the cases $\beta \leq r/(q+r) =0.5$ and $\beta = 1$ (solid lines).
The down- and up-pointing triangles indicate the points at $a^*$ and $b^*$, respectively. The circle on the upper solid line indicates the point at $b^*_c$ of $v_{b^*_c, b^*_c}$ and the one on the lower solid line indicates the point at $a^*_p$ of $v_{a^*_p, \infty}$. (Right) The plots of the value functions $v_{a^*,b^*}$ (dotted) for $r = 0.001$, $0.01$, $0.02$, ...., $0.07$, $0.071$, $0.072$, $0.073$, $0.074$ along with the cases $r \geq q \beta / (1 - \beta) = 0.075$ and $r = 0$ (solid lines).
The down- and up-pointing triangles indicate the points at $a^*$ and $b^*$, respectively. The circle on the lower solid line indicates the point at $b^*_c$ of $v_{b^*_c, b^*_c}$, and the one on the upper solid line indicates the point at $a^*_p$ of $v_{a^*_p, \infty}$.
} \label{plot_convergence}
\end{center}
\end{figure}

%

\appendix

\section{Proofs of Propositions \ref{per-con} and \ref{sin-con} and Theorem \ref{theorem_vf}} \label{appendix_proof} 

In this appendix, we prove Propositions \ref{per-con} and \ref{sin-con} as well as Theorem \ref{theorem_vf}. Define, for $a < b$ and $x \in \R$, 
		\begin{align*}
		\mathcal{I}^{(q,r)}_{a-b}(a-x, \theta)&:=Z_{a-b}^{(q,r)}(a-x, \theta)-W_{a-b}^{(q,r)}(a-x)\frac{Z^{(q+r)}(b-a, \theta)}{W^{(q+r)}(b-a)}, \quad \theta \geq 0,\\
		\mathcal{H}_{a-b}^{(q,r)}(a-x)
		&:= (q+r)^{-1} \Big( \frac {W_{a-b}^{(q,r)}(a-x)} {W^{(q+r)}(b-a)} \big[ r Z^{(q+r)}(b-a) + q \big] + r \big[ Z^{(q)}(a-x) - Z_{a-b}^{(q,r)} (a-x) \big] \Big) \\ 
		&= (q+r)^{-1} \Big(- r \mathcal{I}_{a-b}^{(q,r)} (a-x) + q \frac {W_{a-b}^{(q,r)}(a-x)} {W^{(q+r)}(b-a)} + r Z^{(q)}(a-x)\Big),
		\end{align*}
		with $\mathcal{I}^{(q,r)}_{a-b}(\cdot)$ = $\mathcal{I}^{(q,r)}_{a-b}(\cdot, 0)$ for short.
		Note that these have the following relation with $\mathcal{K}_{a-b}^{(q,r)}$, as in \eqref{def_K}: 
		\begin{align} \label{H_I_K_relation}
		\left(r+\frac{q}{Z^{(q+r)}(b-a)}\right) \frac{\mathcal{I}^{(q,r)}_{a-b}(a-x)}{r+q} = \mathcal{K}_{a-b}^{(q,r)} (a-x) - \mathcal{H}^{(q,r)}_{a-b}(a-x), \quad a<b, x \in \R.
		\end{align}

	For $0 \leq a < b$,  define the hitting times of the processes $X_r^a$  and $U^b$, as defined in Section  \ref{section_model}, as
	\begin{align*}
	\tau_b^+(r) &:= \inf \left\{ t > 0: X_r^a (t) > b\right\}, \quad \tau_0^-(r) := \inf \left\{ t > 0: X_r^a (t) < 0\right\}, \\
	\eta_a^- &:= \inf \left\{ t > 0: U^b (t) < a\right\}.	\end{align*}
Throughout this appendix, let $\mathbf{e}_r = T(1)$ be the first periodic dividend decision time, which is an independent exponential random variable with parameter $r$.

\subsection{Proof of Proposition \ref{per-con}.} \label{proof_per-con}
	Let us denote the left hand side of \eqref{L_p_identity} by $f_p(x,a,b)$.
	By applying the strong Markov property at $\tau_b^+(r)$, noticing that 
	\begin{align} \label{U_r_X_match}
	U_r^{(a,b)}(t) = X_r^a(t), \quad L_p^{(a,b)}(t) = L_p^a(t) \quad \textrm{for } t \in [0, \tau_b^+(r))\quad \textrm{and} \quad U^{(a,b)}_r (\tau_b^+(r)) = b,
	\end{align}
	and using Theorems 3.1 and 3.2 in \cite{APY},
	 we obtain, for $x \geq 0$,
	\begin{align} \label{f_p_recursion}
	\begin{split}
		f_p(x,a,b)&=\E_x\Big(\int_{[0,\tau_b^+(r)\wedge\tau_0^-(r)]}{\rm e}^{-qt} \diff L_p^{a}(t)\Big)+\E_x\left({\rm e}^{-q\tau_b^+(r)};\tau_b^+(r)<\tau_0^-(r)\right)f_p(b,a ,b)\\
		&=\frac{\mathcal{H}^{(q,r)}_{a-b}(a-x)}{\mathcal{H}^{(q,r)}_{a-b}(a)}h^{(q,r)}_{a-b}(a)-h^{(q,r)}_{a-b}(a-x)+\Big(\mathcal{I}^{(q,r)}_{a-b}(a-x)-\frac{\mathcal{H}^{(q,r)}_{a-b}(a-x)}{\mathcal{H}^{(q,r)}_{a-b}(a)}\mathcal{I}^{(q,r)}_{a-b}(a)\Big)f_p(b,a,b).
		\end{split}
	\end{align}
	Here,
	\begin{align*}
 	h^{(q,r)}_{a-b}(a-x)
 	&:= \frac r {q+r}\Big[\overline{Z}^{(q)}(a-x)+\frac{\overline{Z}^{(q+r)}(b-a)}{W^{(q+r)}(b-a)}W^{(q,r)}_{a-b}(a-x)-\overline{Z}^{(q,r)}_{a-b}(a-x) - (a-b) \mathcal{I}^{(q,r)}_{a-b}(a-x)\Big]. 
	\end{align*}
	In contrast,  by the strong Markov property, \eqref{Y_matches}, and the fact that $X$ has no negative jumps,
	\begin{align} \label{f_p_markov}
		f_p(b,a,b)=\E_b\left({\rm e}^{-q(\eta_a^-\wedge \mathbf{e}_r)}\right)f_p(a,a,b)+\E_b\left({\rm e}^{-q\mathbf{e}_r}(U^b (\mathbf{e}_r)-a);\mathbf{e}_r<\eta_a^-\right).
	\end{align}
Here, by the Laplace transform of $\eta_a^-$ given on  page 228 of \cite{K}, 
	\begin{multline}\label{per_div}
		\E_b\left({\rm e}^{-q(\eta_a^-\wedge \mathbf{e}_r)}\right)
		=\E_b\left({\rm e}^{-(q+r)\eta_a^-} \right)+r\E_b\Big(\int_0^{\eta_a^-}{\rm e}^{-(q+r)s} \diff s\Big) \\
		= \frac 1 {r+q} \Big( r + q \E_b [{\rm e}^{- (q+r)\eta_a^-}]\Big)
		=\frac{1}{r+q}\left(r+\frac{q}{Z^{(q+r)}(b-a)}\right),
	\end{multline}
	and, using the resolvent for $U^b$ given in Theorem 8.11 of \cite{K} and integration by parts, we obtain
	\begin{multline*}
		\E_b\left({\rm e}^{-q\mathbf{e}_r}(U^b(\mathbf{e}_r)-a);\mathbf{e}_r<\eta_a^-\right)=r\E_b\left(\int_0^{\eta_a^-}{\rm e}^{-(q+r)s}(U^b(s)-a) \diff s\right)\\
		=-r\int_{a-b}^{0}u\frac{W^{(q+r)}(-u)}{Z^{(q+r)}(b-a)} \diff u =-\frac{r}{Z^{(q+r)}(b-a)}\left((a-b)\overline{W}^{(q+r)}(b-a)+ \int_0^{b-a} \overline{W}^{(q+r)}(y) \diff y\right).
	\end{multline*}
	 Substituting these into \eqref{f_p_markov}, we have 
	\begin{align*}
		f_p(b,a,b)&=-\frac{r \big((a-b)\overline{W}^{(q+r)}(b-a)+\int_0^{b-a} \overline{W}^{(q+r)}(y) \diff y\big) }{Z^{(q+r)}(b-a)} +\frac{1}{r+q}\left(r+\frac{q}{Z^{(q+r)}(b-a)}\right)f_p(a,a,b).
	\end{align*}
	Hence, by  \eqref{f_p_recursion}, for all $x \geq 0$, 
	\begin{multline}\label{per_div_1}
		f_p(x,a,b)
		=\frac{\mathcal{H}^{(q,r)}_{a-b}(a-x)}{\mathcal{H}^{(q,r)}_a(a)}\tilde{h}^{(q,r)}_{a-b}(a)-\tilde{h}^{(q,r)}_{a-b}(a-x) \\
		+\frac{1}{r+q}\Big(\mathcal{I}^{(q,r)}_{a-b}(a-x)-\frac{\mathcal{H}^{(q,r)}_{a-b}(a-x)}{\mathcal{H}^{(q,r)}_{a-b}(a)}\mathcal{I}^{(q,r)}_{a-b}(a)\Big) \left(r+\frac{q}{Z^{(q+r)}(b-a)}\right)f_p(a,a,b),
	\end{multline}
		where 
			\begin{align*}
	\tilde{h}^{(q,r)}_{a-b}(a-x)&:=h^{(q,r)}_{a-b}(a-x)+\mathcal{I}^{(q,r)}_{a-b}(a-x)\frac{r}{Z^{(q+r)}(b-a)}\left((a-b)\overline{W}^{(q+r)}(b-a)+\int_0^{b-a} \overline{W}^{(q+r)}(y) \diff y \right) \\
	&= \frac r {q+r}\Big[\overline{Z}^{(q)}(a-x)-\overline{Z}^{(q,r)}_{a-b}(a-x)  + Z_{a-b}^{(q,r)}(a-x) \frac {\overline{Z}^{(q+r)}(b-a)}  {Z^{(q+r)}(b-a)}\Big].
	\end{align*}
	
	Now, setting $x=a$ and solving for $f_p(a,a,b)$ (using the fact that $\mathcal{H}_{a-b}^{(q,r)}(0) = 1$ and $\mathcal{I}_{a-b}^{(q,r)}(0) = \tilde{h}_{a-b}^{(q,r)}(0) = 0$), we obtain that, by \eqref{H_I_K_relation},
		\begin{align}
		f_p(a,a,b)=\frac{\tilde{h}^{(q,r)}_{a-b}(a)}{\mathcal{K}^{(q,r)}_{a-b}(a)}. \label{f_a_a}
	\end{align}


	Hence, substituting \eqref{f_a_a} into \eqref{per_div_1} and using \eqref{H_I_K_relation} (which holds for $x\in \R$), we obtain that 
	\begin{align*}
		f_p(x,a,b)
		&=\frac{\mathcal{H}^{(q,r)}_{a-b}(a-x)}{\mathcal{H}^{(q,r)}_{a-b}(a)}\tilde{h}^{(q,r)}_{a-b}(a)-\tilde{h}^{(q,r)}_{a-b}(a-x) +\left(\mathcal{K}^{(q,r)}_{a-b}(a-x)-\frac{\mathcal{H}^{(q,r)}_{a-b}(a-x)}{\mathcal{H}^{(q,r)}_{a-b}(a)} \mathcal{K}^{(q,r)}_{a-b}(a)\right)\frac{\tilde{h}^{(q,r)}_{a-b}(a)}{\mathcal{K}^{(q,r)}_{a-b}(a)}\\
		&=\frac{\mathcal{K}^{(q,r)}_{a-b}(a-x)}{\mathcal{K}^{(q,r)}_{a-b}(a)}\tilde{h}^{(q,r)}_{a-b}(a)-\tilde{h}^{(q,r)}_{a-b}(a-x).
	\end{align*}
This is equal to  the right hand side of \eqref{L_p_identity} because
		\begin{align*}
		\tilde{h}^{(q,r)}_{a-b}(a-x) 
		&=k_{a-b}^{(q,r)} (a-x) 
		+ \frac r q \overline{Z}^{(q+r)}(b-a) \mathcal{K}^{(q,r)}_{a-b}(a-x).
	\end{align*}

\subsection{Proof of Proposition \ref{sin-con}.} 
To prove  Proposition \ref{sin-con}, 
we first compute the following identity:
\[
g(x,a,b,\theta):=\E_x\left({\rm e}^{-q\tau_b^{+}(r)+\theta[b-X_r^a(\tau_b^{+}(r))]};\tau_b^+(r)<\tau_0^-(r)\right), \quad q > 0, a < b, \theta \geq 0, x \geq 0.
\]
The $\theta= 0$ case has already been considered in Theorem 3.1 of \cite{APY}. Here, we generalize this using Lemma 2.1 in \cite{LRZ} and Corollary 4.1 in \cite{APY}.
\begin{proposition} \label{prop_joint_laplace}
	For any $q>0$, $0 \leq a<b$, $\theta \geq 0$, and $x \geq 0$,
	\begin{align}\label{Lap_3}
		g(x,a,b,\theta)=\mathcal{I}^{(q,r)}_{a-b}(a-x,\theta)-\frac{\mathcal{H}^{(q,r)}_{a-b}(a-x)}{\mathcal{H}^{(q,r)}_{a-b}(a)} \mathcal{I}^{(q,r)}_{a-b}(a,\theta).
	\end{align}
\end{proposition}
\begin{proof} We prove this proposition in three steps.

(1) For $x\geq a$,  by the strong Markov property and the fact that $X$ has no negative jumps, we obtain
	\begin{align*}
		g(x,a,b,\theta)&=\E_x\left({\rm e}^{-q\tau_b^{+}+\theta[b-X(\tau_b^{+})]};\tau_b^+<\tau_a^{-}\wedge\mathbf{e}_r\right)+\E_x\left({\rm e}^{-q(\tau_a^{-}\wedge\mathbf{e}_r)};\tau_a^{-}\wedge\mathbf{e}_r<\tau_b^+\right)g(a,a,b,\theta).
	\end{align*}
	By \eqref{laplace_in_terms_of_z},
	for $\theta = 0$, 
	\begin{align*}
		\E_x\left({\rm e}^{-q(\tau_a^{-}\wedge\mathbf{e}_r)};\tau_a^{-}\wedge\mathbf{e}_r<\tau_b^+\right)
		&= \E_x\left({\rm e}^{-q \tau_a^{-}};\tau_a^{-}<\tau_b^+, \tau_a^{-}< \mathbf{e}_r\right)+ \E_x\left({\rm e}^{-q \mathbf{e}_r};\tau_a^{-}\wedge \tau_b^{+}>\mathbf{e}_r\right)\\&= \E_x\left({\rm e}^{-(q+r) \tau_a^{-}};\tau_a^{-}<\tau_b^+\right)+ \frac  r {q+r} - \frac r {q+r}\E_x\left({\rm e}^{-(q+r) (\tau_a^{-}\wedge \tau_b^{+})}\right) \\
		&=\frac{r}{r+q}(1-Z^{(q+r)}(b-x))+\frac{W^{(q+r)}(b-x)}{W^{(q+r)}(b-a)}(1+r\overline{W}^{(q+r)}(b-a)).
	\end{align*}
This, together with the second equality in \eqref{laplace_in_terms_of_z}, gives
	\begin{multline} \label{h_decomposition_markov}
		g(x,a,b,\theta)=Z^{(q+r)}(b-x,\theta)-W^{(q+r)}(b-x)\frac{Z^{(q+r)}(b-a,\theta)}{W^{(q+r)}(b-a)} \\
		+\left(\frac{r}{r+q}(1-Z^{(q+r)}(b-x))+\frac{W^{(q+r)}(b-x)}{W^{(q+r)}(b-a)}(1+r\overline{W}^{(q+r)}(b-a))\right)g(a,a,b,\theta).
	\end{multline}
We note that $X_r^a = X$ on $[0, \tau_a^+]$ and $Z^{(q)}(a-y) = 1$ for $y \geq a$.  Hence,  for $x \geq 0$, by the strong Markov property and \eqref{h_decomposition_markov},
	\begin{align*}
	\begin{split}
		g(x,a,b,\theta) &=\E_x\left({\rm e}^{-q\tau_a^+}g(X(\tau_a^+), a,b,\theta);\tau_a^+<\tau_0^-\right) \\
		&= \E_x\left({\rm e}^{-q\tau_a^+}Z^{(q+r)}(b-X(\tau_a^+),\theta) ;\tau_a^+<\tau_0^- \right)
		\\ &-\E_x\left({\rm e}^{-q\tau_a^+} W^{(q+r)}(b-X(\tau_a^+)) ;\tau_a^+<\tau_0^- \right)\frac{Z^{(q+r)}(b-a,\theta)}{W^{(q+r)}(b-a)} \\
		&+\Bigg(\frac{r}{r+q} \left[ \E_x\left({\rm e}^{-q\tau_a^+} Z^{(q)} (a- X(\tau_a^+));\tau_a^+<\tau_0^-\right) -\E_x\left({\rm e}^{-q\tau_a^+}Z^{(q+r)}(b-X(\tau_a^+));\tau_a^+<\tau_0^-\right) \right] \\ &+\frac{\E_x\left({\rm e}^{-q\tau_a^+}W^{(q+r)}(b-X(\tau_a^+)) ;\tau_a^+<\tau_0^-\right)}{W^{(q+r)}(b-a)}(1+r\overline{W}^{(q+r)}(b-a))\Bigg)g(a,a,b,\theta).
		\end{split}
		\end{align*}
Here, using
 Corollary 4.1 in \cite{APY} and Lemma 2.1 in \cite{LRZ}, we have
\begin{align*} 
\begin{split}
		\E_x\left({\rm e}^{-q\tau_a^+}Z^{(q+r)}(b-X(\tau_a^+),\theta);\tau_a^+<\tau_0^- \right) &=Z^{(q,r)}_{a-b}(a-x,\theta)-\frac{W^{(q)}(a-x)}{W^{(q)}(a)}Z^{(q,r)}_{a-b}(a,\theta),  \quad \theta \geq 0, \\
		\E_x\left({\rm e}^{-q\tau_a^+}W^{(q+r)}(b-X(\tau_a^+));\tau_a^+<\tau_0^- \right) &=
		W^{(q,r)}_{a-b}(a-x) -\frac{W^{(q)}(a-x)}{W^{(q)}(a)}W^{(q,r)}_{a-b}(a), \\
		\E_x\left({\rm e}^{-q\tau_a^+} Z^{(q)} (a- X(\tau_a^+));\tau_a^+<\tau_0^-\right) &= Z^{(q)}(a-x)-\frac{W^{(q)}(a-x)}{W^{(q)}(a)}Z^{(q)}(a).
		\end{split}
	\end{align*}
	Hence,
		\begin{multline}  \label{Lap_1}
		g(x,a,b,\theta) 
		= \mathcal{I}_{a-b}^{(q,r)}(a-x, \theta) +\mathcal{H}_{a-b}^{(q,r)}(a-x)g(a,a,b,\theta) \\
		-\frac{W^{(q)}(a-x)}{W^{(q)}(a)} \big[\mathcal{I}_{a-b}^{(q,r)}(a, \theta) +\mathcal{H}_{a-b}^{(q,r)}(a)g(a,a,b,\theta) \big].
		\end{multline}
	
(2)  (i) For the case where $X$ is of bounded variation, by setting $x=a$ in \eqref{Lap_1} and solving for $g(a,a,b,\theta)$ (using $\mathcal{H}_{a-b}^{(q,r)}(0) = 1$, $\mathcal{I}_{a-b}^{(q,r)}(0, \theta) = 0$ and $W^{(q)} (0) > 0$), 
we obtain
	\begin{align}\label{Lap_2}
		g(a,a,b,\theta)=-\frac{\mathcal{I}^{(q,r)}_{a-b}(a,\theta)}{\mathcal{H}^{(q,r)}_{a-b}(a)}.
	\end{align}
	
(ii) 
Now suppose $X$ is of unbounded variation.  To show that we obtain the same identity \eqref{Lap_2}, it suffices to modify the proof of (3.21) in \cite{APY} (which considers the $\theta = 0$ case).  Define the spectrally negative \lev process $\tilde{X}:=-X$, and construct the dual  process $\tilde{X}_r^0$ with the lower Parisian reflection barrier $0$, as in \cite{APY}, and its hitting times $\tilde{\tau}_{c}^-(r) := \inf \{ t> 0: \tilde{X}_r^0(t) < c\}$ and $\tilde{\tau}_{c}^+(r) := \inf\{ t > 0: \tilde{X}_r^0(t) > c\}$ for $c \in \R$.  Then, we have
	\begin{align*}
	g(a,a,b,\theta)
	=\E\left({\rm e}^{-q\tilde{\tau}_{a-b}^-(r)+\theta(\tilde{X}_r^0(\tilde{\tau}_{a-b}^-(r))+b-a)};\tilde{\tau}_{a-b}^-(r)<\tilde{\tau}_{a}^+(r) \right).
	\end{align*}
	
	Now, following the proof of (3.21) from Section 5  in \cite{APY}, we arrive at
	\begin{align}
	g(a,a,b,\theta)=\frac{W^{(q)}(b)}{\mathcal{H}^{(q,r)}_{a-b}(a)} \mathcal{N} \label{h_decomposition_unbounded}
	\end{align}
	with
		\begin{align*}
\mathcal{N} :=\textbf{n}\left({\rm e}^{-q(\tilde{\tau}_0^- + \tilde{\tau}_{a-b}^- \circ \Theta_{\tilde{\tau}_0^-})+\theta( \tilde{X}(\tilde{\tau}_0^- +\tilde{\tau}_{a-b}^-+\Theta_{\tilde{\tau}_0^-})+b-a)};\tilde{\tau}_{0}^-< \tilde{\tau}_{a}^+,\tilde{\tau}_{a-b}^- \circ\Theta_{\tilde{\tau}_0^-}<\mathbf{e}_r\wedge (\tilde{\tau}_{0}^+ \circ\Theta_{\tilde{\tau}_0^-})\right),
	\end{align*}
	where $\tilde{\tau}^+$ and $\tilde{\tau}^-$ are defined analogously to $\tau^+$ and $\tau^-$ for $\tilde{X}$ and  $\Theta$ is the time-shift operator and   $\textbf{n}$ denotes the excursion measure of the spectrally negative  L\'evy process $\tilde{X}$ away from zero (see \cite{PPR15b} for more details). Here, 
	\begin{align*}
	\mathcal{N}
	&=\textbf{n}\left({\rm e}^{-q\tilde{\tau}_0^-}\E_{\tilde{X}(\tilde{\tau}_0^-)}\left({\rm e}^{-q\tilde{\tau}_{a-b}^-+\theta(\tilde{X}(\tilde{\tau}_{a-b}^-)+b-a)}; \tilde{\tau}_{a-b}^- <\mathbf{e}_r \wedge \tilde{\tau}_{0}^+ \right);\tilde{\tau}_0^-< \tilde{\tau}_{a}^+ \right) \\
		&= \textbf{n} \left({\rm e}^{-q \tilde{\tau}_0^-}\left(Z^{(q+r)}(\tilde{X}(\tilde{\tau}_0^-)+b-a,\theta)-\frac{W^{(q+r)}(\tilde{X}(\tilde{\tau}_0^-)+b-a)}{W^{(q+r)}(b-a)}Z^{(q+r)}(b-a,\theta)\right);\tilde{\tau}_0^-<\tilde{\tau}_{a}^+\right)\\
	&=-\frac{1}{W^{(q)}(b)}\left(Z_{a-b}^{(q,r)}(a,\theta)-\frac{Z^{(q+r)}(b-a,\theta)}{W^{(q+r)}(b-a)}W^{(q,r)}_{a-b}(a)\right)=-\frac{\mathcal{I}_{a-b}^{(q,r)}(a,\theta)}{W^{(q)}(b)},
	 \end{align*}		
	 where we have used the strong Markov property for the first equality, \eqref{laplace_in_terms_of_z} for the second equality and a slight modification of Lemma 6.3 in \cite{YP2016b} for the last equality. By substituting this into \eqref{h_decomposition_unbounded}, we obtain \eqref{Lap_2}.

(3) Now,  by using \eqref{Lap_2}  in \eqref{Lap_1}, the proof is complete for both the bounded and unbounded variation cases.

\end{proof}

Using the proposition above, we have the following.
\begin{corollary} \label{corollary_overshoot}
	For any $q>0$, $0 \leq  a<b$, and $x \geq 0$,
	\begin{align*}
		\E_x\left({\rm e}^{-q\tau_b^{+}(r)}[X_r^a(\tau_b^{+}(r))-b];\tau_b^+(r)<\tau_0^-(r)\right)=\frac{\mathcal{H}^{(q,r)}_{a-b}(a-x)}{\mathcal{H}^{(q,r)}_{a-b}(a)}\mathcal{J}^{(q,r)}_{a-b}(a)-\mathcal{J}^{(q,r)}_{a-b}(a-x),
	\end{align*}
	where
	\begin{multline*}
		\mathcal{J}^{(q,r)}_{a-b}(a-x):=\overline{Z}^{(q,r)}_{a-b}(a-x)-\psi'(0+)\overline{W}^{(q,r)}_{a-b}(a-x) \\ -\frac{W^{(q,r)}_{a-b}(a-x)}{W^{(q+r)}(b-a)}\left(\overline{Z}^{(q+r)}(b-a)-\psi'(0+)\overline{W}^{(q+r)}(b-a)\right).
	\end{multline*}
\end{corollary}
\begin{proof}
	Because 
		$\partial_\theta Z^{(q+r)}(y,\theta)\big|_{\theta=0+}=\overline{Z}^{(q+r)}(y)-\psi'(0+)\overline{W}^{(q+r)}(y)$  for $y \in \R$ (see the proof of Corollary 3.5 in \cite{YP2016b}), we have
	\begin{align}\label{der_0_Zv}
		\partial_\theta \mathcal{I}^{(q,r)}_{a-b}(a-x,\theta)\big|_{\theta=0 +}= \mathcal{J}^{(q,r)}_{a-b}(a-x).
	\end{align}
	The result now follows by differentiating \eqref{Lap_3}, setting $\theta=0 +$, and using \eqref{der_0_Zv}.
\end{proof}
\subsubsection{Proof of Proposition \ref{sin-con}.}
First, denote the left hand side of \eqref{L_c_identity} by $f_c(x,a,b)$.
	By applying the strong Markov property at $\tau_b^+(r)$,   Proposition \ref{prop_joint_laplace},  Corollary \ref{corollary_overshoot}, and \eqref{U_r_X_match}, we have, for $x \geq 0$,
	\begin{align} \label{f_c_markov_1}
	\begin{split}
		f_c&(x,a,b)=\E_x\left({\rm e}^{-q\tau_b^{+}(r)}[X_r^a(\tau_b^{+}(r))-b +f_c(b,a,b)];\tau_b^+(r)<\tau_0^-(r)\right) \\ 
		&=\frac{\mathcal{H}^{(q,r)}_{a-b}(a-x)}{\mathcal{H}^{(q,r)}_{a-b}(a)}\mathcal{J}^{(q,r)}_{a-b}(a)-\mathcal{J}^{(q,r)}_{a-b}(a-x)+\Big(\mathcal{I}^{(q,r)}_{a-b}(a-x)-\frac{\mathcal{H}^{(q,r)}_{a-b}(a-x)}{\mathcal{H}^{(q,r)}_{a-b}(a)}\mathcal{I}^{(q,r)}_{a-b}(a)\Big)f_c(b,a,b).
		\end{split}
	\end{align}
	On the other hand, by applying the strong Markov property at  $\eta_a^- \wedge \mathbf{e}_r$ and because $L_c^{(a,b)}=L_c^b$ on $[0, \eta_a^-\wedge\mathbf{e}_r]$, and $X$ does not have negative jumps, 
	\begin{align} \label{f_c_markov}
		f_c(b,a,b)&=\E_b\Big(\int_0^{\eta_a^-\wedge\mathbf{e}_r}{\rm e}^{-qt} \diff L_{c}^{b}(t)\Big)+\E_b\left({\rm e}^{-q(\eta_a^-\wedge\mathbf{e}_r)}\right)f_c(a,a,b).
	\end{align}
Here, as in the proof of Theorem 1 in \cite{APP2007}, the first term on the right hand side is 
	\begin{align}\label{sin_div_1}
		\E_b\Big(\int_0^{\eta_a^-\wedge\mathbf{e}_r}{\rm e}^{-qt} \diff L^{b}_{c}(t)\Big)=\E_b\Big(\int_0^{\eta_a^-}{\rm e}^{-(q+r)t} \diff L^{b}_{c}(t)\Big)=\frac{\overline{Z}^{(q+r)}(b-a)- \psi'(0+) \overline{W}^{(q+r)}(b-a)}{Z^{(q+r)}(b-a)}.
	\end{align}
	Hence, using \eqref{per_div} and \eqref{sin_div_1} in \eqref{f_c_markov},
	\begin{align*}
		f_c(b,a,b)=\frac{\overline{Z}^{(q+r)}(b-a)- \psi'(0+) \overline{W}^{(q+r)}(b-a)}{Z^{(q+r)}(b-a)}+\frac{1}{r+q}\left(r+\frac{q}{Z^{(q+r)}(b-a)}\right)f_c(a,a,b).
	\end{align*}
	Substituting this in \eqref{f_c_markov_1} and using \eqref{H_I_K_relation} (which holds for all $x\in\R$), \begin{align}\label{sin_div_2}
	\begin{split}
		f_c(x,a,b)&=\frac{\mathcal{H}^{(q,r)}_{a-b}(a-x)}{\mathcal{H}^{(q,r)}_{a-b}(a)} \tilde{i}^{(q,r)}_{a-b}(a)- \tilde{i}^{(q,r)}_{a-b}(a-x) \\
		&+\left(\mathcal{I}^{(q,r)}_{a-b}(a-x)-\frac{\mathcal{H}^{(q,r)}_{a-b}(a-x)}{\mathcal{H}^{(q,r)}_{a-b}(a)}\mathcal{I}^{(q,r)}_{a-b}(a)\right)  \frac{1}{r+q}\left(r+\frac{q}{Z^{(q+r)}(b-a)}\right)f_c(a,a,b) \\
		&=\frac{\mathcal{H}^{(q,r)}_{a-b}(a-x)}{\mathcal{H}^{(q,r)}_{a-b}(a)} \tilde{i}^{(q,r)}_{a-b}(a)- \tilde{i}^{(q,r)}_{a-b}(a-x) +\left(\mathcal{K}^{(q,r)}_{a-b}(a-x) -\frac{\mathcal{H}^{(q,r)}_{a-b}(a-x)}{\mathcal{H}^{(q,r)}_{a-b}(a)}\mathcal{K}^{(q,r)}_{a-b}(a)\right)  f_c(a,a,b),
		\end{split}
	\end{align}
		where 
	\begin{align*}
		\tilde{i}_{a-b}^{(q,r)}(a-x)&:=-\mathcal{I}^{(q,r)}_{a-b}(a-x) \frac{\overline{Z}^{(q+r)}(b-a)- \psi'(0+) \overline{W}^{(q+r)}(b-a)}{Z^{(q+r)}(b-a)} +\mathcal{J}^{(q,r)}_{a-b}(a-x) \\
		&=\overline{Z}^{(q,r)}_{a-b}(a-x)-\psi'(0+)\overline{W}^{(q,r)}_{a-b}(a-x)-\frac{Z^{(q,r)}_{a-b}(a-x)}{Z^{(q+r)}(b-a)}\left(\overline{Z}^{(q+r)}(b-a)-\psi'(0+)\overline{W}^{(q+r)}(b-a)\right). 
		\end{align*}
%
Hence setting $x=a$ and solving for $f_c(a,a,b)$ (using $\mathcal{H}_{a-b}^{(q,r)}(0) =\mathcal{K}_{a-b}^{(q,r)}(0) = 1$ and $\mathcal{I}_{a-b}^{(q,r)}(0) = \tilde{i}_{a-b}^{(q,r)}(0) = 0$) gives us that $f_c(a,a,b)= {\displaystyle \tilde{i}_{a-b}^{(q,r)}(a)}  / {\mathcal{K}_{a-b}^{(q,r)}(a)}$.
Substituting this  in \eqref{sin_div_2}, we obtain
	\begin{align*}
		f_c(x,a,b)
		&=\frac{\mathcal{K}^{(q,r)}_{a-b}(a-x)}{\mathcal{K}^{(q,r)}_{a-b}(a)}\tilde{i}^{(q,r)}_{a-b}(a)-\tilde{i}^{(q,r)}_{a-b}(a-x).
	\end{align*}
	This equals the right hand side of \eqref{L_c_identity} because 
	\begin{align*}
		\tilde{i}_{a-b}^{(q,r)}(a-x)	
		= i_{a-b}^{(q,r)} (a-x) -\frac {(r+q) \mathcal{K}^{(q,r)}_{a-b}(a-x) } q\left(\overline{Z}^{(q+r)}(b-a)-\psi'(0+)\overline{W}^{(q+r)}(b-a)\right).
	\end{align*}
\subsection{Proof of Theorem \ref{theorem_vf}.}
	\par Using Propositions \ref{per-con} and \ref{sin-con}, 
	\begin{align}
		v_{a,b}(x)
		&=\frac{\mathcal{K}^{(q,r)}_{a-b}(a-x)}{\mathcal{K}^{(q,r)}_{a-b}(a)}\left(k^{(q,r)}_{a-b}(a)+\beta i^{(q,r)}_{a-b}(a)\right)-\left(k^{(q,r)}_{a-b}(a-x)+\beta i^{(q,r)}_{a-b}(a-x)\right). \label{v_a_b_sum}
	\end{align}
	To show that this reduces to \eqref{vf_ff}, it is sufficient to show that
			\begin{align} \label{k_relation}
			k^{(q,r)}_{a-b}(a-x)+\beta i^{(q,r)}_{a-b}(a-x)&= \Gamma(a,b;x)-\frac{\beta\psi'(0+)}{q}\mathcal{K}_{a-b}^{(q,r)} (a-x)Z^{(q+r)}(b-a);
		\end{align}
		substituting this in \eqref{v_a_b_sum} and after simplification, we obtain \eqref{vf_ff}.
	
We have
		\begin{multline} \label{diff_h_g_Gamma}
			k^{(q,r)}_{a-b}(a-x)+\beta i^{(q,r)}_{a-b}(a-x) - \Gamma(a,b,x) \\
			= - \frac {\beta \psi'(0+)} {q}\left(q\overline{W}_{a-b}^{(q,r)}(a-x)+rZ^{(q)}(a-x)\overline{W}^{(q+r)}(b-a)+ 1\right).
		\end{multline}
		By \eqref{W_a_def},
		$Z_{a-b}^{(q,r)} (a-x) = (q+r) \overline{W}_{a-b}^{(q,r)} (a-x) + 1 - r \overline{W}^{(q)}(a-x)$, and hence
		\begin{align}\label{aux_1}
			\frac 1 {q+r} \Big[ Z_{a-b}^{(q,r)} (a-x) - 1 + r \overline{W}^{(q)}(a-x) \Big] = \overline{W}_{a-b}^{(q,r)} (a-x).
		\end{align}
		Now, using (\ref{aux_1}), it is straightforward to obtain
		\begin{align*}
			q&\overline{W}_{a-b}^{(q,r)}(a-x)+rZ^{(q)}(a-x)\overline{W}^{(q+r)}(b-a)\\&=\frac{q}{q+r}\Big[ Z_{a-b}^{(q,r)} (a-x) - 1 + r \overline{W}^{(q)}(a-x) \Big]+rZ^{(q)}(a-x)\overline{W}^{(q+r)}(b-a)\\
			&=\mathcal{K}_{a-b}^{(q,r)} (a-x)Z^{(q+r)}(b-a)-1.
		\end{align*}
		
		Substituting this into \eqref{diff_h_g_Gamma}, we have \eqref{k_relation}, as desired.
		
\section{Proof of Lemma \ref{lemma_derivatives_Gamma_K}} \label{proof_lemma_derivatives}

For $a < b$ and $x \in \R$, integration by parts gives
\begin{align*} 
\begin{split}
		\overline{Z}_{a-b}^{(q,r)}(a-x) 
		&=\overline{Z}^{(q+r)}(b-x) \\ &- r \Big[   \overline{Z}^{(q+r)}(b-a) \overline{W}^{(q)} (a-x) + \int_{0}^{a-x}  \overline{W}^{(q)}(a-x-y) Z^{(q+r)}(y-a+b)\diff y \Big], \\
		Z_{a-b}^{(q,r)}(a-x) 
		&=Z^{(q+r)}(b-x) \\ &- r \Big[   Z^{(q+r)}(b-a) \overline{W}^{(q)} (a-x) + (q+r) \int_{0}^{a-x}  \overline{W}^{(q)}(a-x-y) W^{(q+r)}(y-a+b)\diff y \Big].
		\end{split}
\end{align*}

Hence, for $x \neq a$, 
\begin{align*}
		&\frac \partial {\partial x}\overline{Z}_{a-b}^{(q,r)}(a-x) \\ &=-Z^{(q+r)}(b-x) + r \Big[    \overline{Z}^{(q+r)}(b-a) W^{(q)} (a-x) + \int_{0}^{a-x}  W^{(q)}(a-x-y) Z^{(q+r)}(y-a+b)\diff y \Big]  \\
		&=-Z^{(q+r)}(b-x) + r \Big[    \overline{Z}^{(q+r)}(b-a) W^{(q)} (a-x)  + Z^{(q+r)}(b-a) \overline{W}^{(q)} (a-x) \\ &+ (q+r) \int_{0}^{a-x}  \overline{W}^{(q)}(a-x-y) W^{(q+r)}(y-a+b)\diff y \Big],
\end{align*}
and
\begin{align*} 
\begin{split}
		\frac \partial {\partial x}&Z_{a-b}^{(q,r)}(a-x) =-(q+r) W^{(q+r)}(b-x) + r \Big[    Z^{(q+r)}(b-a) W^{(q)} (a-x) \\ &+ (q+r) \int_{0}^{a-x}  W^{(q)}(a-x-y) W^{(q+r)}(y-a+b)\diff y \Big]  \\
		&=-(q+r) W^{(q+r)}(b-x) + r \Big[   Z^{(q+r)}(b-a) W^{(q)} (a-x) \\ & + (q+r) \Big( W^{(q+r)}(b-a) \overline{W}^{(q)} (a-x) + \int_{0}^{a-x}  \overline{W}^{(q)}(a-x-y) W^{(q+r) \prime}(y-a+b)\diff y  \Big)\Big].
\end{split}
\end{align*}
Hence, we have the first equalities of \eqref{H_derivatives} and \eqref{Gamma_derivatives}. By differentiating these further, the rest of the results are immediate.


	\end{document}